\documentclass{amsart}
\usepackage{graphicx}
\usepackage{dsfont}
\usepackage{fancyhdr}
\usepackage{helvet}
\usepackage[lined,algonl,boxed]{algorithm2e}
\newtheorem{thm}{Theorem}[section]

\newtheorem{lem}[thm]{Lemma}

\theoremstyle{definition}
\newtheorem{defn}[thm]{Definition}
\theoremstyle{remark}
\newtheorem{rem}[thm]{Remark}

\numberwithin{equation}{section}
\newcommand{\norm}[1]{\left\Vert#1\right\Vert}
\newcommand{\abs}[1]{\left\vert#1\right\vert}
\newcommand{\babs}[1]{\Big \vert#1 \Big \vert}
\newcommand{\set}[1]{\left\{#1\right\}}
\newcommand{\parr}[1]{\left (#1\right )}
\newcommand{\brac}[1]{\left [#1\right ]}
\newcommand{\ip}[1]{\left \langle #1 \right \rangle }
\newcommand{\Real}{\mathbb R}
\newcommand{\eps}{\varepsilon}

\newcommand{\too}{\rightarrow}

\newcommand{\OO}{\mathcal{O}}
\newcommand{\bbar}[1]{\overline{#1}}
\newcommand{\wt}[1]{\widetilde{#1}} 
\newcommand{\wh}[1]{\widehat{#1}} 

\def \M{\mathcal{M}}
\def \N{\mathcal{N}}

\def \PL{PL} 
\def \ncPL{ncPL} 
\def \z{{z_0}} 
\def \w{{w_0}}

\def \bfi{\textbf{\footnotesize{i}}} 

\def \C{\mathbb{C}} 
\def \D{\mathcal{D}} 
\def \Md{M_D} 

\def \vol{\mbox{\rm{\footnotesize{vol}}}} 
\def \Vol{\mbox{\rm{vol}}} 

\def \V{V}

\def \E{E}
\def \F{F}
\def \mM{\textbf{\textsf{M}}}
\def \mV{\textbf{\textsf{V}}}

\def \mE{\textbf{\textsf{E}}}
\def \mE{\textbf{\textsf{E}}}
\def \mF{\textbf{\textsf{F}}}
\def \mv{\textbf{\textsf{v}}}

\def \me{\textbf{\textsf{e}}}
\def \mf{\textbf{\textsf{f}}}%

\def \R{\mathbb{R}}
\def \ddelta{\mbox{\boldmath{$\delta$}}}

\def \Mbs{M\"{o}bius }
\def \d{\mbox{\bf{d}}}
\def \X{\mathcal{X}}
\def \Y{\mathcal{Y}}
\def \P{\mathcal{P}}
\def \MM{\mathbb{M}}
\def \GG{\mathbb{G}}
\def \stu{*u}

\begin{document}
\title[Surface Comparison with Mass Transportation]{Surface Comparison with Mass Transportation}
\author{Y. Lipman, I. Daubechies }
\address{Princeton University}
%
%
\begin{abstract}
We use mass-transportation as a tool
to compare surfaces (2-manifolds). In particular, we
determine the ``similarity''  of two given surfaces by solving
a mass-transportation problem between their conformal densities.
This mass transportation problem differs from the standard case
in that we require the solution to be invariant under global
M\"{o}bius transformations.\\
Our approach provides a constructive way of defining a metric in the abstract
space of simply-connected smooth surfaces with boundary
(i.e. surfaces of disk-type);
this metric can also be used to define meaningful intrinsic
distances between pairs of ``patches''  in the two surfaces,
which allows automatic alignment of the surfaces.
We provide numerical experiments on ``real-life'' surfaces to demonstrate
possible applications in natural sciences.

\end{abstract}
\maketitle

\section{introduction}
Alignment of surfaces plays a role in a wide range of scientific
disciplines. It is a standard problem in comparing different scans
of manufactured objects; various algorithms have been proposed for
this purpose in the computer graphics literature. It is  often also
a crucial step in a variety of problems in medicine and biology; in
these cases the surfaces tend to be more complex, and the alignment
problem may be harder. For instance, neuroscientists studying brain
function through functional Magnetic Resonance Imaging (fMRI)
typically observe several people performing identical tasks,
obtaining readings for the corresponding activity in the brain
cortex of each subject. In a first approximation, the cortex can be
viewed as a highly convoluted 2-dimensional surface. Because
different cortices are folded in very different ways, a synthesis of
the observations from different subjects must be based on
appropriate mappings between pairs of brain cortex surfaces, which
reduces to a family of surface alignment problems
\cite{Fischl99,Haxby09}. In another example, paleontologists
studying molar teeth of mammals rely on detailed comparisons of the
geometrical features of the tooth surfaces to distinguish species or
to determine similarities or differences in diet \cite{Jukka07}.

Mathematically, the problem of surface alignment can be described as
follows: given two 2-surfaces $\M$ and $\N$, find a mapping $f:\M
\rightarrow \N$ that preserves, as best possible, ``important
properties'' of the surfaces. The nature of the ``important
properties'' depends on the problem at hand. In this paper, we
concentrate on preserving the geometry, i.e., we would like the map
$f$ to preserve intrinsic distances, to the extent possible. In
terms of the examples listed above, this is the criterion
traditionally selected in the computer graphics literature; it also
corresponds to the point of view of the paleontologists studying
tooth surfaces. To align cortical surfaces, one typically uses the
Talairach method \cite{Lancaster00} (which relies on geometrically
defined landmarks and is thus geometric in nature as well), although
alignment based on functional correspondences has been proposed more
recently \cite{Haxby09}.

In this paper we propose a procedure to ``geometrically'' align
surfaces, based on uniformization theory and optimal mass
transportation. This approach is related to the computer graphics
constructions in \cite{Lipman:2009:MVF}, which rely on the
representation of isometries between topologically equivalent
simply-connected surfaces by M\"{o}bius transformations between
their uniformization spaces, and which exploit that 1) the
M\"{o}bius group has small dimensionality (e.g. 3 for disk-type
surfaces and 6 for sphere-type) and 2) changing the metric in one
piece of a surface has little influence on the uniformization of
distant parts. These two observations lead, in
\cite{Lipman:2009:MVF}, to fast and particularly effective
algorithms to identify near-isometries between differently deformed
versions of a surface. In our present context, these same
observations lead to a simple algorithm for surface alignment,
reducing it to a linear programming problem.

We shall restrict ourselves to (sufficiently smooth) disk-type
surfaces; we map them to metric densities defined on the hyperbolic
disk, their canonical uniformization space. (Apart from simplifying
the description of the surface, this also removes any effect of
global translations and rotations on the description of each
individual surface.)  The alignment problem can then be studied in
the framework of Kantorovich mass-transportation
\cite{Kantorovich1942} between these metric densities, as follows.
Mass-transportation seeks to minimize the ``average distance'' over
which mass needs to be ``moved'' (in the most efficient such moving
procedure) to transform one mass density $\mu$ into another, $\nu$.
In our case the uniformizing metric density (or conformal factor)
corresponding to an initial surface is not unique, but is defined
only up to a M\"{o}bius transformation. Because a na\"{\i}ve
application of mass-transportation  on the hyperbolic disk would not
possess the requisite invariance under M\"{o}bius transformations,
we generalize the mass-transportation framework, and replace the
metric $d(x,y)$ traditionally used in defining the ``average
displacement distance'' by a metric that depends on $\mu$ and $\nu$,
measuring the dissimilarity between the two metric densities on
neighborhoods of $x$ and $y$. Introducing neighborhoods also makes
the definition less sensitive to noise in practical applications.
The optimal way of transporting mass in this generalized framework,
in which the orientation in space of the original surfaces is
``factored away'', automatically defines a corresponding optimal way
of aligning the surfaces.

Our approach also allows us to define a new distance between
surfaces. The average distance over which mass needs transporting
(to transform one metric density into the other) quantifies the
extent to which the two surfaces differ; we prove that it defines a
distance metric between surfaces.

Other distances between surfaces have been used recently for several
applications \cite{memoli07}. A prominent mathematical approach to
define distances between surfaces considers the surfaces as special
cases of \emph{metric spaces}, and uses then the Gromov-Hausdorff
(GH) distance between metric spaces \cite{Gromov06}. The GH distance
between metric spaces $X$ and $Y$ is defined through examining all
the isometric embedding of $X$ and $Y$ into (other) metric spaces;
although this distance possesses many attractive mathematical
properties, it is inherently hard computationally
\cite{memoli05,BBK06}. For instance, computing the GH distance is
equivalent to a non-convex quadratic programming problem; solving
this directly for correspondences is equivalent to integer quadratic
assignment, and is thus NP-hard \cite{Cela98}. In addition, the
non-convexity implies that the solution found in practice may be a
local instead of a global minimum, and is therefore not guaranteed
to give the correct answer for the GH distance. The distance metric
between surfaces that we define in this paper does not have these
shortcomings: because the computation of the distance between
surfaces in our approach can be recast as a linear program, it can
be implemented using efficient polynomial algorithms that are
moreover guaranteed to converge to the correct solution.

It should be noted that in \cite{memoli07}, Memoli generalizes the
GH distance of \cite{memoli05} by introducing a quadratic mass
transportation scheme to be applied to metric spaces already
equipped with a measure (mm spaces); he notes that the computation
of this Gromov-Wasserstein distance for mm spaces is somewhat easier
and more stable to implement than the original GH distance. In our
approach we do not need to equip the surfaces we compare with a
measure: after uniformization reduces the problem to comparing two
disks, we naturally "inherit" two corresponding conformal factors
that we interpret as measure densities, for which we then apply an
approach similar to the one proposed in \cite{memoli07}. Another
crucial aspect in which our work differs from \cite{memoli07} is
that, in contrast to the (continuous) quadratic programming method
proposed in \cite{memoli07} to compute the Gromov-Wasserstein
distance between mm spaces, our conformal approach leads to a convex
(even linear) problem, solvable via a linear programming method.

It is worth mentioning that optimal mass transportation has been
used as well, in the engineering literature to define interesting
metrics between images; in this context metric is often called the
Wasserstein distance. The seminal work for this image analysis
approach is the paper by Rubner et al.~\cite{Rubner2000-TEM}, in
which images are viewed as discrete measures, and the distance is
called appropriately the ``Earth Mover's Distance''.

Another related method is presented in the papers of Zeng et al.
\cite{Gu2008_a,Gu2008_b}, which also use the uniformization space to
match surfaces. Our work differs from that of Zeng et al. in that
they use prescribed feature points (defined either by the user or by
extra texture information) to calculate an interpolating harmonic
map between the uniformization spaces, and then define the final
correspondence as a composition of the uniformization maps and this
harmonic interpolant. This procedure is highly dependent on the
prescribed feature points, provided as extra data or obtained from
non-geometric information. In contrast, our work does not use any
prescribed feature points, or external data, and makes use of only
the geometry of the surface; in particular we make use of the
conformal structure itself to define deviation from (local)
isometry.

Our paper is organized as follows: in Section \ref{s:prelim} we
briefly recall some facts about uniformization and optimal mass
transportation that we shall use, at the same time introducing our
notation. Section \ref{s:optimal_vol_trans_for_surfaces} contains
the main results of this paper, constructing the distance metric
between disk-type surfaces, in several steps. Section
\ref{s:the_discrete_case_implementation} discusses various issues
that concern the numerical implementation of the framework we
propose; Section \ref{s:examples} illustrates our results with a few
examples.

\section{Background and Notations}
\label{s:prelim}

As described in the introduction, our framework makes use of two
mathematical theories: uniformization theory, to represent the
surfaces as measures defined on a canonical domain, and optimal mass
transportation, to align the measures. In this section we recall
some of their basic properties, and we introduce our notations.

\subsection{Uniformization}

By the celebrated uniformization theory for Riemann surfaces (see
for example \cite{Springer57,Farkas92}), any simply-connected
Riemann surface is conformally equivalent to one of three canonical
domains: the sphere, the complex plane, or the unit disk. Since
every 2-manifold surface $\M$ equipped with a smooth Riemannian
metric $g$ has an induced conformal structure and is thus a Riemann
surface, uniformization applies to such surfaces. Therefore, every
simply- connected surface with a Riemannian metric can be mapped
conformally to one of the three canonical domains listed above. We
shall consider surfaces $\M$ that are topologically equivalent to
disks and that come equipped with a Riemannian metric tensor $g$
(possibly inherited from the standard 3D metric if the surface is
embedded in $\R^3$). For each such $\M$  there exists a conformal
map $\phi:\M \rightarrow \D$, where $\D =\{z \ | \ |z|<1\}$ is the
open unit disk. The map $\phi$ pushes $g$ to a metric on $\D$;
denoting the coordinates in $\D$ by $z=x^1+\bfi  x^2$, we can write
this metric as
$$
\wt{g} = \phi_* g = \widetilde{\mu}(z)\, \delta_{ij}\, dx^i \otimes
dx^j,
$$
where $\widetilde{\mu}(z)>0$, Einstein summation convention is used,
and the subscript $*$ denotes the ``push-forward'' action. The
function $\widetilde{\mu}$ can also be viewed as the \emph{density
function} of the measure $\Vol_\M$ induced by the Riemann volume
element: indeed, for (measurable) $A \subset \M$,
\begin{equation}\label{e:volume_element}
    \Vol_\M(A) = \int_{\phi(A)} \widetilde{\mu}(z) \, dx^1\wedge dx^2.
\end{equation}

It will be convenient to use the hyperbolic metric on the unit disk
$(1-|z|^2)^{-2}\delta_{ij} dx^i \otimes dx^j$ as a reference metric,
rather than the standard Euclidean $\delta_{ij} dx^i \otimes dx^j$;
note that they are conformally equivalent (with conformal factor
$(1-|z|^2)^{-2}$). Instead of the density $\widetilde{\mu}(z)$, we
shall therefore use the {\em hyperbolic density function}
\begin{equation}\label{e:relation_hyperbolic_euclidean_density}
\mu^H(z):=(1-|z|^2)^{2}\,\widetilde{\mu}(z)\,,
\end{equation}
where the superscript $H$ stands for hyperbolic. We shall often drop
this superscript: unless otherwise stated $\mu=\mu^H$, and
$\nu=\nu^H$ in what follows. The density function $\mu=\mu^H$
satisfies
$$\Vol_\M(A) = \int_{\phi(A)} \mu(z)\, d\vol_H(z)\,,$$
where $d\vol_H(z)=(1-|z|^2)^{-2}\, dx^1\wedge dx^2$.

The conformal mappings of $\D$ to itself are the disk-preserving
M\"{o}bius transformations $m \in \Md$, a family with three real
parameters, defined by
\begin{equation}\label{e:disk_mobius}
    m(z) = e^{\bfi \theta}\frac{z-a}{1-\bar{a}z}, \ a\in \D, \ \theta \in [0,2\pi).
\end{equation}
Since these M\"{o}bius transformations satisfy
\begin{equation}\label{e:disk_mobius_is_isometry_of_hyperbolic_geom}
    (1-|m(z)|^2)^{-2}|m'(z)|^2 = (1-|z|^2)^{-2} \,,
\end{equation}
where $m'$ stands for the derivatives of $m$, the pull-back of $\mu$
under a mapping $m\in \Md$ takes on a particularly simple
expression. Setting $w=m(z)$, with $w=y^1+\bfi  y^2$, and
$\widetilde{g}(w)=\widetilde{\mu}(w)\delta_{ij}dy^i\otimes dy^j =
\mu(w) (1-|w|^2)^{-2}\delta_{ij}dy^i\otimes dy^j$, the definition
\[
(m^*\widetilde{g})(z)_{kl}\,dx^{k}\otimes dx^{\ell} :=
\mu(w)\,(1-|w|^2)^{-2}\,\delta_{ij} \, dy^{i}\otimes dy^{j}
\]
implies
\begin{align*}
(m^*\widetilde{g})_{k\ell}(z)\,dx^{k}\otimes dx^{\ell} &=\mu(m(z))(1-|m(z)|^2)^{-2}\,\delta_{ij}\,\frac{\partial y^i}{\partial x^k}\,\frac{\partial y^j}{\partial x^\ell} \,dx^{k}\otimes dx^{\ell}\\
&= \mu(m(z))\,(1-|m(z)|^2)^{-2}\,|m'(z)|^2\,\delta_{k\ell} \,dx^{k}\otimes dx^{\ell}\\
&=\mu(m(z))\,(1-|z|^2)^{-2}\,\delta_{k\ell}\,dx^{k}\otimes dx^{\ell}.\\
\end{align*}
In other words, $(m^*\widetilde{g})(z)_{kl}\,dx^{k}\otimes
dx^{\ell}$ takes on the simple form
$m^*\mu(z)\,(1-|z|^2)^{-2}\,\delta_{kl}\,dx^{k}\otimes dx^{\ell}$,
with
\begin{equation}\label{e:pullback_of_metric_density_mu_by_mobius}
    m^*\mu(z) = \mu(m(z)).
\end{equation}
Likewise, the push-forward, under a disk M\"{o}bius transform
$m(z)=w$, of the diagonal Riemannian metric defined by the density
function $\mu=\mu^H$, is again a diagonal metric, with (hyperbolic)
density function $m_{*}\mu (w)=\left(m_{*}\mu \right)^H(w)$ given by
\begin{equation}\label{e:push_forward_of_metric_density}
    m_* \mu(w) = \mu(m^{-1}(w)).
\end{equation}

It follows that checking whether or not two surfaces $\M$ and $\N$
are isometric, or searching for (near-)\ isometries between $\M$ and
$\N$, is greatly simplified by considering the conformal mappings
from $\M$, $\N$ to $\D$: once the (hyperbolic) density functions
$\mu$ and $\nu$ are known, it suffices to identify $m \in \Md$ such
that $\nu(m(z))$ equals $\mu(z)$ (or ``nearly'' equals, in a sense
to be made precise). This was exploited in \cite{Lipman:2009:MVF} to
construct fast algorithms to find corresponding points between two
given surfaces.

\subsection{Optimal mass transportation}

Optimal mass transportation was introduced by G. Monge
\cite{Monge1781}, and L. Kantorovich \cite{Kantorovich1942}. It
concerns the transformation of one mass distribution into another
while minimizing a cost function that can be viewed as the amount of
work required for the task. In the Kantorovich formulation, to which
we shall stick in this paper, one considers two measure spaces
$X,Y$,  a probability measure on each, $\mu \in P(X)$, $\nu \in
P(Y)$ (where $P(X),P(Y)$ are the respective probability measure
spaces on $X$ and $Y$), and the space $\Pi(\mu,\nu)$ of probability
measures on $X \times Y$ with marginals $ \mu$ and $\nu$ (resp.),
that is, for $A\subset X$, $B\subset Y$, $\pi(A\times Y) = \mu(A)$
and $\pi(X \times B) = \nu(B)$. The \emph{optimal} mass
transportation is the element of $\Pi(\mu,\nu)$ that minimizes
$\int_{X \times Y}d(x,y)d\pi(x,y)$, where $d(x,y)$ is a cost
function. (In general, one should consider an infimum rather than a
minimum; in our case, $X$ and $Y$ are compact, $d(\cdot,\cdot)$ is
continuous, and the infimum is achieved.) The corresponding minimum,
\begin{equation}\label{e:basic_Kantorovich_transporation}
    T^R_d(\mu,\nu) = \mathop{\inf}_{\pi \in \Pi(\mu,\nu)}\int_{X \times Y}d(x,y)d\pi(x,y),
\end{equation}
is the optimal mass transportation distance between $\mu$ and $\nu$,
with respect to the cost function $d(x,y)$.

Intuitively, one can interpret this as follows: imagine being
confronted with a pile of sand on the one hand ($\mu$), and a hole
in the ground on the other hand ($-\nu$), and assume that the volume
of the sand pile equals exactly the volume of the hole (suitably
normalized, $\mu,\nu$ are probability measures). You wish to fill
the hole with the sand from the pile ($\pi \in \Pi(\mu,\nu)$), in a
way that minimizes the amount of work (represented by $\int
d(x,y)d\pi(x,y)$, where $d(\cdot,\cdot)$ can be thought of as a
distance function). In the engineering literature, the distance
$T^R_d(\mu,\nu)$ is often called the ``earth mover's distance''
\cite{Rubner2000-TEM}, a name that echoes this intuition.

In what follows, we shall apply this framework to the density
functions $\mu$ and $\nu$ on the hyperbolic disk $\D$ obtained by
conformal mappings from two surfaces $\M$, $\N$, as described in the
previous subsection.

The main obstacle to applying the Kantorovich transportation
framework directly is that the density $\mu$, characterizing the
Riemannian metric on $\D$ obtained by pushing forward the metric on
$\M$ via the uniformizing map $\phi:\M \rightarrow \D$, is not
uniquely defined: another uniformizing map $\phi':\M \rightarrow \D$
may well produce a different $\mu'$. Because the two representations
are necessarily isometric ($\phi^{-1} \circ \phi'$ maps $\M$
isometrically to itself), we must have $\mu'(m(z))=\mu(z)$ for some
$m \in \Md$. (In fact, $m=\phi' \circ \phi^{-1}$.) In a sense, the
representation of (disk-type) surfaces $\M$ as measures over $\D$
should be considered ``modulo'' the disk M\"{o}bius transformations.

We thus need to address how to adapt the optimal transportation
framework to factor out this M\"{o}bius transformation ambiguity.
This is done by designing a special distance (or cost) functional
$d^R_{\mu,\nu}(z,w)$ that {\em depends} on the conformal densities
$\mu$ and $\nu$ representing the two surfaces. (A fairly simple
argument shows that a cost function that does not depend on $\mu$
and $\nu$ allows only trivial answers, such as $d(z,w)=0$ for all
$z,w$.) As we shall see in the next section,
this cost function will have an intuitive explanation:
$d^R_{\mu,\nu}(z,w)$ will measure how well an $R$-sized neighborhood
of $z$ with density $\mu$ can be matched isometrically to an
$R$-sized neighborhood of $w$ with density $\nu$ by means of a disk
M\"{o}bius transformation.

\section{Optimal volume transportation for surfaces}
\label{s:optimal_vol_trans_for_surfaces} We want to measure
distances between surfaces by using the Kantorovich transportation
framework to measure the transportation between the metric densities
on $\D$ obtained by uniformization applied to the surfaces. The main
obstacle is that these metric densities are not uniquely defined;
they are defined up to a M\"{o}bius transformation. In particular,
if two densities $\mu$ and $\nu$ are related by $\nu=m_*\mu$ (i.e.
$\mu(z)=\nu(m(z))$), where $m \in \Md$, then we want our putative
distance between $\mu$ and $\nu$ to be zero, since they describe
isometric surfaces, and could have been obtained by different
uniformization maps of the same surface. A standard approach to
obtain quantities that are invariant under the operation of some
group (in our case, the disk M\"{o}bius transformations) is by
minimizing over the possible group operations. For instance, we
could set
\[
\mbox{Distance}(\mu,\nu)=\inf_{m \in \Md} \left( \inf_{\pi \in
\Pi(m_*\mu, \nu)}\,\int_{\D \times \D} d(z,w)\,d\pi(z,w)\,\right)\,,
\]
where $\Pi(\mu, \nu)$ is the set of probability measures on $\D
\times \D$ with marginals $\mu \,\vol_H$ and $\nu \,\vol_H$. In
order for this to be computationally feasible, we would want the
minimum to be achieved in some $m$, which would depend on $\mu$ and
$\nu$ of course; let's denote this special minimizing $m \in \Md$ by
$m_{\mu,\nu}$. This would mean
\begin{align}
\mbox{Distance}(\mu,\nu)&=  \inf_{\pi \in \Pi([m_{\mu,\nu}]_*\mu,
\nu)}\,\int_{\D \times \D}
d(z,w)\,d\pi(z,w)\nonumber\\
&=\inf_{\pi \in \Pi(\mu, \nu)}\,\int_{\D \times \D}
d(m_{\mu,\nu}(z),w)\,d\pi(z,w)\,.\label{id_Dist}
\end{align}
If $\nu$ were itself already equal to $m'_*\mu$, for some $m' \in
\Md$, then we would expect the minimizing M\"{o}bius transformation
to be $m_{\mu,\nu}=m'$; for $\pi$ supported on the diagonal
$\verb"d"=\{(z,z)\,;\,z \in \D\}\subset \D \times \D$, defined by
$\pi(A)=\int_{A_2} \nu(w)\, d\vol_H(w)$, with $A_2= \{w; (w,w) \in
A\}$, one would then indeed have $\int_{\D \times
\D}d(m_{\mu,\nu}(z),w)\,d\pi(z,w)=0$, leading to
$\mbox{Distance}(\mu,m'_*\mu)=0$. From (\ref{id_Dist}) one sees that
this amounts to using the same formula as for the standard
Kantorovich approach with just one change: {\em the cost function
depends on} $\mu$ and $\nu$.

We shall use a variant on this construction, retaining the principle
of using  cost functions $d(\cdot,\cdot)$ in the integrand that
depend on $\mu$ and $\nu$, without picking them necessarily of the
form $d(m_{\mu,\nu}(z),w)$.  In addition to introducing such a
dependence, we also wish to incorporate some robustness into the
evaluation of the distance between (or dissimilarity of) $\mu$ and
$\nu$. We shall do this by using a cost function
$d^R_{\mu,\nu}(z,w)$ that depends on a comparison of the behavior
$\mu$ and $\nu$ on {\em neighborhoods} of $z$ and $w$, mapped by $m$
ranging over $\Md$. The next subsection shows precisely how this is
done.

\subsection{Construction of $d^R_{\mu,\nu}(z,w)$}

We construct $d^R_{\mu,\nu}(z,w)$ so that it indicates the extent to
which a neighborhood of the point $z$ in $(\D,\mu)$, the (conformal
representation of the) first surface, is isometric with a
neighborhood of the point $w$ in $(\D,\nu)$, the (conformal
representation of the) second surface. We will need to define two
ingredients for this: the neighborhoods we will use, and how we
shall characterize the (dis)similarity of two neighborhoods,
equipped with different metrics.

We start with the neighborhoods.

For a fixed radius $R>0$, we define $\Omega_{z_0,R}$ to be the
hyperbolic geodesic disk of radius $R$ centered at $z_0$. The
following gives an easy procedure to construct these disks. If
$z_0=0$, then  the hyperbolic geodesic disks centered at $z_0=0$ are
also  ``standard'' (i.e. Euclidean) disks centered at 0:
$\Omega_{0,R} = \{z \,;\, |z|\leq r_R \}$, where
$r_R=\mbox{arctanh}(r)=R$. The hyperbolic disks around other centers
are images of these central disks under M\"{o}bius transformations
(= hyperbolic isometries): setting
$m(z)=(z-z_0)(1-z\bar{z_0})^{-1}$, we have
\begin{equation}\label{e:neighborhood_def}
    \Omega_{z_0,R} = m^{-1}(\Omega_{0,R})\,.
\end{equation}
If $m'$, $m''$ are two maps in $\Md$ that both map $z_0$ to 0, then
$m'' \circ (m')^{-1}$ simply rotates $\Omega_{0,R}$ around its
center, over some angle $\theta$ determined by $m'$ and $m''$. From
this observation one easily checks that (\ref{e:neighborhood_def})
holds for {\em any} $m \in \Md$ that maps $z_0$ to $0$. In fact, we
have the following more general
\begin{lem}\label{lem:m(omega_z)=omega_w}
For arbitrary $z,w \in \D$ and any $R>0$, every disk M\"{o}bius
transformation $m\in \Md$ that maps $z$ to $w$ (i.e. $w=m(z)$) also
maps $\Omega_{z,R}$ to $\Omega_{w,R}$.
\end{lem}

Next we define how to quantify the (dis)similarity of the pairs
$\left(\Omega_{z_0,R}\,,\, \mu\,\right)$ and
$\left(\Omega_{w_0,R}\,,\, \nu\,\right)$. Since (global) isometries
are given by the elements of the disk-preserving M\"{o}bius group
$\Md$, we will test the extent to which the two patches are
isometric by comparing $\left(\Omega_{w_0,R}\,,\, \nu\,\right)$ with
all the images of $\left(\Omega_{z_0,R}\,,\, \mu\,\right)$ under
M\"{o}bius transformations in $\Md$ that take $z_0$ to $w_0$.

To carry out this comparison, we need a norm. Any metric $g_{ij}(z)
dx^i \otimes dx^j$ induces an inner product on the space of
2-covariant tensors, as follows: if $\mathbf{a}(z) = a_{ij}(z)
\,dx^i \otimes dx^j$ and $\mathbf{b}(z) = b_{ij}(z) \,dx^i \otimes
dx^j$ are two 2-covariant tensors in our parameter space $\D$, then
their inner product is defined by
\begin{equation}\label{e:inner_product_2-covariant_tensor}
    \langle \mathbf{a}(z), \mathbf{b}(z)\rangle =
a_{ij}(z)\,b_{k\ell}(z)\,g^{ik}(z)\,g^{j\ell}(z)~;
\end{equation}
as always, this inner product defines a norm, $\|\mathbf{a}\|_z^2 =
a_{ij}(z)\,a_{k\ell}(z)\,g^{ik}(z)\,g^{j\ell}(z)$.

Now, let us apply this to the computation of the norm of the
difference between the local metric on one surface,
$g_{ij}(z)=\mu(z)(1-|z|^2)^{-2}\delta_{ij}$, and
$h_{ij}(w)=\nu(w)(1-|w|^2)^{-2}\delta_{ij}$, the pull-back metric
from the other surface by a M\"{o}bius transformation $m$. Using
(\ref{e:inner_product_2-covariant_tensor}),(\ref{e:pullback_of_metric_density_mu_by_mobius}),
and writing $\ddelta$ for the tensor with entries $\delta_{ij}$, we
have:
\begin{align*}
\|\mu - m^*\nu\|_{z}^2 & = \|\,\mu(z) (1-|z|^2)^{-2}\ddelta -
\nu(m(z))
(1-|z|^2)^{-2} \ddelta\,\|_{z}^2 \\
& = \Big(\mu(z) -
\nu(m(z))\Big)^2(1-|z|^2)^{-4}\,\delta_{ij}\,\delta_{k\ell}
\,g^{ik}(z)\,g^{j\ell}(z)=\left(1 -
\frac{\nu(m(z))}{\mu(z)}\right)^2.
\end{align*}

We are now ready to define the distance function
$d^R_{\mu,\nu}(z,w)$:
\begin{equation}\label{e:d_mu,nu(z,w)_def}
    d^R_{\mu,\nu}(z_0,w_0) :=
\mathop{\inf}_{m \in \Md\,,\,m(z_0)=w_0}\int_{\Omega_{z_0,R}}
\,|\,\mu(z) - (m^*\nu)(z)\,|\, d\vol_H(z),
\end{equation}
where $d\vol_H(z)=(1-|z|^2)^{-2} \,dx\wedge dy$ is the volume form
for the hyperbolic disk. The integral in (\ref{e:d_mu,nu(z,w)_def})
can also be written in the following form, which makes its
invariance more readily apparent:
\begin{equation}\label{e:d_invariant_form}
    \int_{\Omega_{z_0,R}}\left|\,1 - \frac{\nu(m(z))}{\mu(z)}\right| \,d\vol_\M(z) = \int_{\Omega_{z_0,R}} \|\mu - m^*\nu \|_z \, d\vol_\M(z),
\end{equation}
where $d\vol_\M(z)=\mu(z)(1-|z|^2)^{-2}\,dx^1\wedge dx^2
=\sqrt{|g_{ij}|}\,dx^1\wedge dx^2$ is the volume form of the first
surface $\M$.

The next Lemma shows that although the integration in
(\ref{e:d_invariant_form}) is carried out w.r.t. the volume of the
first surface, this measure of distance is nevertheless symmetric:
\begin{lem}\label{lem:symmetry_of_d_integral}
If  $m\in \Md$ maps $z_0$ to $w_0$, $m(z_0)=w_0$, then
$$
\int_{\Omega_{z_0,R}}\Big|\,\mu(z) - m^*\nu(z) \,\Big|\, d\vol_H(z)
= \int_{\Omega_{w_0,R}}\Big|\,m_*\mu(w) - \nu(w) \, \Big|\,
d\vol_H(w).
$$
\end{lem}
\begin{proof}
By the pull-back formula
(\ref{e:pullback_of_metric_density_mu_by_mobius}), we have
$$
\int_{\Omega_{z_0,R}}\Big|\,\mu(z) - m^*\nu(z) \,\Big|\, d\vol_H(z)
= \int_{\Omega_{z_0}}\Big|\,\mu(z) - \nu(m(z)) \,\Big|\, d\vol_H(z).
$$
Performing the change of coordinates $z=m^{-1}(w)$ in the integral
on the right hand side, we obtain
$$
\int_{m(\Omega_{z_0,R})}\, \Big|\,\mu(m^{-1}(w)) - \nu(w) \Big|\,
d\vol_H(w),
$$
where we have used that $m^{-1}$ is an isometry and therefore
preserves the volume element $d\vol_H(w)=(1-|w|^2)^{-2} \,dy^1
\wedge dy^2$. By Lemma \ref{lem:m(omega_z)=omega_w},
$\,m(\Omega_{z_0,R})=\Omega_{w_0,R}\,$; using the push-forward
formula (\ref{e:push_forward_of_metric_density}) then allows to
conclude.
\end{proof}

Note that our point of view in defining our ``distance'' between $z$
and $w$ differs from the classical point of view in mass
transportation: Traditionally, $d(z,w)$ is some sort of
\emph{physical distance} between the points $z$ and $w$; in our case
$d^R_{\mu,\nu}(z,w)$ measures the dissimilarity of (neighborhoods
of) $z$ and $w$.

The next Theorem lists some important properties of $d^R_{\mu,\nu}$;
its proof is given in Appendix A.
\begin{thm}\label{thm:properties_of_d}
The distance function $d^R_{\mu,\nu}(z,w)$ satisfies the following
properties
\begin{table}[ht]
\begin{tabular}{c l l}
{\rm (1)} & $~d^R_{m^*_1\mu,m^*_2\nu}(m^{-1}_1(\z),m^{-1}_2(\w)) = d^R_{\mu,\nu}(\z,\w)~$ & {\rm Invariance under (well-defined)}\\
& & {\rm M\"{o}bius changes of coordinates} \\
&&\\
{\rm (2)} & $~d^R_{\mu,\nu}(\z,\w) = d^R_{\nu,\mu}(\w,\z)~$ & {\rm Symmetry} \\
&&\\
{\rm (3)} & $~d^R_{\mu,\nu}(\z,\w) \geq 0~$ & {\rm Non-negativity} \\
&&\\
{\rm (4)} &\multicolumn{2}{c}{$\!\!\!\!\!\!d^R_{\mu,\nu}(\z,\w) = 0 \,\Longrightarrow \, \Omega_{z_0,R}$ {\rm in} $(\D,\mu)$ {\rm and} $\Omega_{w_0,R}$ {\rm in} $(\D,\nu)$ {\rm are isometric} }\\
&&\\
{\rm (5)} & $~d^R_{m^*\nu, \nu}(m^{-1}(\z),\z)=0~$ & {\rm Reflexivity} \\
&&\\
{\rm (6)} & $~d^R_{\mu_1,\mu_3}(z_1,z_3) \leq
d^R_{\mu_1,\mu_2}(z_1,z_2) + d^R_{\mu_2,\mu_3}(z_2,z_3)~$ & {\rm
Triangle inequality}
\end{tabular}
\end{table}

\end{thm}

In addition, the function
$d^R_{\mu,\nu}:\D\times\D\,\rightarrow\,\R$ is continuous. To show
this, we first look a little more closely at the family of disk
M\"{o}bius transformations that map one pre-assigned point $z_0 \in
\D$ to another pre-assigned point $w_0 \in \D$, over which one
minimizes to define $d_{\mu}^R(z_0,w_0)$.

\begin{defn}\label{def:M_D,z_0,w_0}
For any pair of points $z_0,\,w_0 \in \D$, we denote by
$M_{D,z_0,w_0}$ the set of M\"{o}bius transformations that map $z_0$
to $w_0$.
\end{defn}

This family of M\"{o}bius transformations is completely
characterized by the following lemma:

\begin{lem}\label{lem:a_and_tet_formula_in_mobius_interpolation}
For any $z_0,w_0 \in \D$, the set $M_{D,z_0,w_0}$ constitutes a
$1$-parameter family of disk M\"{o}bius transformations,
parametrized continuously over $S^1$ (the unit circle). More
precisely, every $m \in M_{D,z_0,w_0}$ is of the form
\begin{equation}\label{e:a_of_mobius}
m(z)= \tau\,\frac{z-a}{1-\overline{a}z}~,~~\mbox{ {\rm{with} }}~~ a
= a(z_0,w_0,\sigma) :=\frac{z_0-w_0
\,\overline{\sigma}}{1-\overline{z_0}\,w_0\,\overline{\sigma}}
~~~\mbox{{\rm and }}~~ \tau = \tau(z_0,w_0,\sigma) := \sigma
\frac{1- \overline{z_0} \,w_0 \,\overline{\sigma}}
{1-z_0\,\overline{w_0}\, \sigma},
\end{equation}
where $\sigma \in S_1:=\{z \in \C\,;\,|z|=1\}$ can be chosen freely.
\end{lem}
\begin{proof}
By (\ref{e:disk_mobius}), the disk M\"{o}bius transformations that
map $z_0$ to $0$ all have the form
\[
m_{\psi,z_0}(z)=e^{\bfi \psi}\,\frac{z-z_0}{1-\overline{z_0}\,z}\,,
~\mbox{ the inverse of which is }~~ m_{\psi,z_0}^{-1}(w)=e^{-\bfi
\psi}\, \frac{w+e^{\bfi \psi}z_0}{1+ e^{-\bfi
\psi}\,\overline{z_0}w}~,
\]
where $\psi \in \R$ can be set arbitrarily. It follows that the
elements of $M_{D,z_0,w_0}$ are given by the family
$m_{\gamma,w_0}^{-1}\circ m_{\psi,z_0}$, with $\psi,\,\gamma \in
\R$. Working this out, one finds that these combinations of
M\"{o}bius transformations take the form (\ref{e:a_of_mobius}), with
$\sigma=e^{\bfi (\psi-\gamma)}$.
\end{proof}
We shall denote by $m_{z_0,w_0,\sigma}$ the special disk M\"{obius}
transformation defined by (\ref{e:a_of_mobius}). In view of our
interest in $d^R_{\mu,\nu}$, we also define the auxiliary function
\[
\Phi: \D \times \D \times S_1 \longrightarrow \C
\]
by $\Phi(z_0,w_0,\sigma) =
\int_{\Omega(z_0,R)}\,|\,\mu(z)-\nu(m_{z_0,w_0,\sigma}(z))\,|\,d\vol_H(z)$.
This function has the following continuity properties, inherited
from $\mu$ and $\nu$:

\begin{lem}\label{lem:auxiliary} $~$\\
$\bullet$ For each fixed $(z_0,w_0)$, the function $\Phi(z_0,w_0,\cdot)$ is continuous on $S_1$.\\
$\bullet$ For each fixed $\sigma \in S_1$, $
\Phi(\cdot,\cdot,\sigma) $ is continuous on $\D \times \D$.
Moreover, the family $ \Big(\Phi(\cdot,\cdot,\sigma)\Big)_{\sigma
\in S_1} $ is equicontinuous.
\end{lem}
\begin{proof}
The proof of this Lemma is given in Appendix A.
\end{proof}

Note that since $S^1$ is compact, Lemma \ref{lem:auxiliary} implies
that the infimum in the definition of $d^R_{\mu,\nu}$ can be
replaced by a minimum:
\[
d^R_{\mu,\nu}(z_0,w_0)=\mathop{\min}_{m(z_0)=w_0}\,
\int_{\Omega_{z_0,R}}\,|\,\mu(z)-\nu(m(z))\,|\,d\vol_H(z)~.
\]

We have now all the building blocks to prove
\begin{thm}\label{thm:continuity_of_d^R_mu,nu}
If $\mu$ and $\nu$ are continuous from $\D$ to $\R$, then
$d^R_{\mu,\nu}(z,w)$ is a continuous function on 
$\D\times\D$.
\end{thm}
\begin{proof}
Pick an arbitrary point  $(z_0,w_0) \in \D \times \D$, and pick
$\eps>0$ arbitrarily small.

By Lemma \ref{lem:auxiliary}, there exists a $\delta>0$ such that,
for $|z'_0-z_0|<\delta$, $|w'_0-w_0|<\delta$, we have
\[
\left|\,\Phi(z_0,w_0,\sigma)-\Phi(z'_0,w'_0,\sigma)\,\right|\,\leq\,
\eps~,
\]
uniformly in $\sigma$. Pick now arbitrary $z'_0,w'_0$ so that
$|z_0-z'_0|,|w_0-w'_0|<\delta$.

Let $m_{z_0,w_0,\sigma}$, resp. $m_{z'_0,w'_0,\sigma'}$, be the
minimizing \Mbs transform in the definition of
$d_{\mu,\nu}^R(z_0,w_0)$, resp. $d_{\mu,\nu}^R(z'_0,w'_0)$, i.e.
\[
d^R_{\mu,\nu}(z_0,w_0)= \Phi(z_0,w_0,\sigma) \ \ \textrm{and} \ \
d^R_{\mu,\nu}(z'_0,w'_0)= \Phi(z_0,w_0,\sigma')~.
\]

It then follows that
\begin{align*}
d^R_{\mu,\nu}(z_0,w_0)&=\min_{\tau}\Phi(z_0,w_0,\tau)
\leq \Phi(z_0,w_0,\sigma')\\
&\leq \Phi(z'_0,w'_0,\sigma')+ |\Phi(z_0,w_0,\sigma') -
\Phi(z'_0,w'_0,\sigma')|= d^R_{\mu,\nu}(z'_0,w'_0)+
|\Phi(z_0,w_0,\sigma') - \Phi(z'_0,w'_0,\sigma')|\\
&\leq d^R_{\mu,\nu}(z'_0,w'_0)+ \mathop{\sup}_{\omega \in
S_1}|\Phi(z_0,w_0,\omega) - \Phi(z'_0,w'_0,\omega)| \leq
d^R_{\mu,\nu}(z'_0,w'_0)+ \eps~.
\end{align*}
Likewise $d^R_{\mu,\nu}(z'_0,w'_0) \leq d^R_{\mu,\nu}(z_0,w_0) +
\eps$, so that $\abs{d^R_{\mu,\nu}(z_0,w_0) -
d^R_{\mu,\nu}(z'_0,w'_0)}<\eps$.
\end{proof}

\subsection{Incorporating $d^R_{\mu,\nu}(z,w)$ into the transportation framework}
The next step in constructing the distance operator between surfaces
is to incorporate the distance $d^R_{\mu,\nu}(z,w)$ defined in the
previous subsection into the (generalized) Kantorovich
transportation model:
\begin{equation}
T^R_d(\mu,\nu)=\inf_{\pi\in \Pi(\mu,\nu)}\int_{\D\times
\D}d^R_{\mu,\nu}(z,w)d\pi(z,w).
\label{e:generalized_Kantorovich_transportation}
\end{equation}
The main result is that this procedure (under some extra conditions)
furnishes a \emph{metric} between (disk-type) surfaces.

\begin{thm}
There exists $\pi^* \in \Pi(\mu,\nu)$ such that
$$\int_{\D\times \D}d^R_{\mu,\nu}(z,w)d\pi^*(z,w)=\inf_{\pi\in \Pi(\mu,\nu)}\int_{\D\times \D}d^R_{\mu,\nu}(z,w)d\pi(z,w).$$
\end{thm}

\begin{proof}
This proof follows the same argument as in \cite{Villani:2003},
adapted here to our generalized setting. It uses the continuity of
the distance function to derive the existence of a global minimum of
(\ref{e:generalized_Kantorovich_transportation}). Let
$\Big(\pi_k\Big)_{k\in \mathbb{N}} \in \Pi(\mu,\nu)$ be a minimizer
sequence of (\ref{e:generalized_Kantorovich_transportation}), for
example by taking
$$
\int_{\D\times \D}d^R_{\mu,\nu}(z,w)d\pi_k(z,w) < T^R_d(\mu,\nu) +
\frac{1}{k}.
$$
Then this sequence of measures is tight, that is, for every $\eps
>0$, there exists a compact set $C\subset \D\times\D$ such that
$\pi_k(C)>1-\eps$, for all $k \in \mathbb{N}$. To see this, note
that since $\D$ is separable and complete, the measures $\mu$, $\nu$
are \emph{tight} measures (see \cite{Billingsley68}). This means
that for arbitrary $\eps >0$, there exist compact sets $A,B \subset
\D$ so that $\mu(A)>1-\eps/2$ and $\nu(B)>1-\eps/2$. It then follows
that, for all $k \in \mathbb{N}$,
$$
\pi_k(A\times B) = \pi_k(A\times \D) - \pi_k(A \times (\D \setminus
B)\,) \geq \mu(A) - \nu(\D\setminus B)  = \mu(A) - (1 - \nu(B)) > 1
- \eps.
$$
Since the set $C = A \times B \subset \D\times \D$ is compact, this
proves the claimed tightness of the family $\Big(\pi_k\Big)_{k\in
\mathbb{N}}$. By Prohorov's Theorem \cite{Billingsley68}, a tight
family of measures is sequentially weakly compact; in our case this
means that $\Big(\pi_k\Big)_{k\in \mathbb{N}}$ has a weakly
convergent subsequence $\Big(\pi_{k_n}\Big)_{n\in \mathbb{N}}$; by
definition, its weak limit $\pi^*$ satisfies, for every bounded
continuous function $f$ on $\D\times\D$,
$$
\int_{\D\times \D}f(z,w)d\pi_{k_n}(z,w) \rightarrow \int_{\D\times
\D}f(z,w)d\pi^*(z,w).
$$
Therefore, taking in particular the continuous function
$f(z,w)=d^R_{\mu,\nu}(z,w)$, we obtain
$$
T^R_d(\mu,\nu) = \lim_{n\rightarrow \infty} \int_{\D\times
\D}f(z,w)d\pi_{k_n}(z,w) = \int_{\D\times \D}f(z,w)d\pi^*(z,w).
$$

\end{proof}

Under rather mild conditions, the ``standard'' Kantorovich
transportation (\ref{e:basic_Kantorovich_transporation}) on a metric
spaces $(X,d)$  defines a metric on the space of probability
measures on $X$ . We will prove that our generalization defines a
distance metric as well. More precisely, we shall prove first that
$$
\d^R(\M,\N)=T^R_d(\mu,\nu)
$$
defines a semi-metric in the set of all disk-type surfaces. We shall
restrict ourselves to surfaces that are sufficiently smooth to allow
uniformization, so that they can be globally and conformally
parameterized over the hyperbolic disk. Under some extra
assumptions, we will prove that $\d^R$ is a metric, in the sense
that $\d^R(\M,\N)=0$ implies that $\M$ and $\N$ are isometric.

For the semi-metric part we will again adapt a proof given in
\cite{Villani:2003} to our framework. In particular, we shall make
use of the following ``gluing lemma'':
\begin{lem}
\label{lem:gluing_lemma} Let $\mu_1,\mu_2,\mu_3$ be three
probability measures on $\D$, and let $\pi_{12} \in
\Pi(\mu_1,\mu_2)$, $\pi_{23} \in \Pi(\mu_2,\mu_3)$ be two
transportation plans. Then there exist a probability measure $\pi$
on $\D \times \D \times \D$ that has $\pi_{12},\pi_{23}$ as
marginals, that is $\int_{z_3\in\D} d\pi(z_1,z_2,z_3) =
d\pi_{12}(z_1,z_2) $, and $\int_{z_1\in\D} d\pi(z_1,z_2,z_3) =
d\pi_{23}(z_2,z_3)$.
\end{lem}
This lemma will be used in the proof of the following:

\begin{thm}
For two disk-type surfaces $\M=(\D,\mu)$, $\N=(\D,\nu)$, let
$\d^R(\M,\N)$ be defined by
$$
\d^R(\M,\N)=T^R_d(\mu,\nu).
$$
Then $\d^R$ defines a semi-metric on the space of disk-type
surfaces.
\end{thm}
\begin{proof}

The symmetry of $d^R_{\mu,\nu}$ implies symmetry for $T^R_d$, by the
following argument:
\begin{align*}
T^R_d(\mu,\nu) &=
\mathop{\inf}_{\pi \in \Pi(\mu,\nu)}\int_{\D \times \D}d^R_{\mu,\nu}(z,w)d\pi(z,w) = \mathop{\inf}_{\pi \in \Pi(\mu,\nu)}\int_{\D \times \D}d^R_{\nu,\mu}(w,z)d\pi(z,w) \\
&=
\mathop{\inf}_{\pi \in \Pi(\mu,\nu)}\int_{\D \times \D}d^R_{\nu,\mu}(w,z)d\widetilde{\pi}(w,z),  ~~~~~~~\mbox{ where we have set }~\widetilde{\pi}(w,z)=\pi(z,w)\\
&= T^R_d(\nu,\mu)~.  ~~~~~~~~~(\mbox{ use that }\pi\in\Pi(\mu,\nu)
\Leftrightarrow \widetilde{\pi}\in\Pi(\nu,\mu))
\end{align*}


The non-negativity of $d^R_{\mu,\nu}(\cdot,\cdot)$ automatically
implies $T^R_d(\mu,\nu) \geq 0$.


Next we show that, for any \Mbs transformation $m$,
$T^R_d(\mu,m_*\mu)=0$. To see this, pick the transportation plan
$\pi \in \Pi(\mu,m_*\mu)$ defined by
$$
\int_{\D\times \D} f(z,w) d\pi(z,w) = \int_{\D} f(z,m(z))
\mu(z)\,d\vol_H(z).
$$
On the one hand $\pi \in \Pi(\mu,m_*\mu)$, since
$$\int_{A \times \D} d\pi(z,w) = \int_{A} \mu(z)d\vol_H(z),$$
and
\begin{align*}
\int_{\D\times B} d\pi(z,w) &=
\int_{\D\times \D} \chi_B(w) d\pi(z,w) \\
&=\int_{\D} \chi_B(m(z)) \mu(z)d\vol_H(z) = \int_{\D} \chi_B(w)
\mu_*(w)d\vol_H(w),
\end{align*}
where we used the change of variables $w=m(z)$ in the last step.
Furthermore, $\pi(z,w)$ is concentrated on the graph of $m$, i.e. on
$\set{(z,m(z)) \ ; \ z\in \D} \subset \D \times \D$. Since
$d^R_{\mu,m_*\mu}(z,m(z)) = 0$ for all $z \in \D$ we obtain
therefore $T_d(\mu,m_*\mu) \leq \int_{\D\times \D}
d^R_{\mu,m_*\mu}(z,w)d\pi(z,w)  = 0$.


Finally, we prove the triangle inequality $T^R_d(\mu_1,\mu_3) \leq
T^R_d(\mu_1,\mu_2) + T^R_d(\mu_2,\mu_3)$ . To this end we follow the
argument in the proof given in \cite{Villani:2003} (page 208). This
is where we invoke the gluing Lemma stated above.

We start by picking  arbitrary transportation plans $\pi_{12} \in
\Pi(\mu_1,\mu_2)$ and $\pi_{23} \in \Pi(\mu_2,\mu_3)$. By Lemma
\ref{lem:gluing_lemma} there exists a probability measure $\pi$ on
$\D\times \D \times \D$ with marginals $\pi_{12}$ and $\pi_{23}$.
Denote by $\pi_{13}$ its third marginal, that is
$$\int_{z_2\in \D}d\pi(z_1,z_2,z_3) = d\pi_{13}(z_1,z_3).$$
Then
\begin{align*}
T^R_d(\mu_1,\mu_3) &\leq \int_{\D \times \D} d^R_{\mu_1,\mu_3}(z_1,z_3)d\pi_{13}(z_1,z_3)  = \int_{\D \times \D \times \D} d^R_{\mu_1,\mu_3}(z_1,z_3)d\pi(z_1,z_2,z_3) \\
&\leq \int_{\D \times \D \times \D} \Big( d^R_{\mu_1,\mu_2}(z_1,z_2) + d^R_{\mu_2,\mu_3}(z_2,z_3) \Big )d\pi(z_1,z_2,z_3) \\
&\leq \int_{\D \times \D \times \D} d^R_{\mu_1,\mu_2}(z_1,z_2)
d\pi(z_1,z_2,z_3) +
\int_{\D \times \D \times \D} d^R_{\mu_2,\mu_3}(z_2,z_3) d\pi(z_1,z_2,z_3) \\
&\leq \int_{\D \times \D } d^R_{\mu_1,\mu_2}(z_1,z_2)
d\pi_{12}(z_1,z_2) + \int_{\D \times \D } d^R_{\mu_2,\mu_3}(z_2,z_3)
d\pi_{23}(z_2,z_3),
\end{align*}
where we used the triangle-inequality for $d^R_{\mu,\nu}$ listed in
(Theorem \ref{thm:properties_of_d}). Since we can choose $\pi_{12}$
and $\pi_{23}$ to achieve arbitrary close values to the infimum in
eq.~(\ref{e:generalized_Kantorovich_transportation}) the triangle
inequality follows.
\end{proof}

To qualify as a metric rather than a semi-metric, $\d^R$ (or
$T^R_d$) should be able to distinguish from each other any two
surfaces (or measures) that are not ``identical'', that is
isometric. To prove that they can do so, we need an extra
assumption: we shall require that the surfaces we consider have no
self-isometries. More precisely, we require that each surface $\M$
that we consider satisfies the following definition:

\begin{defn}
A surface $\M$ is said to be a singly
$\varrho\mbox{-}_{\mbox{\tiny{H}}}\mbox{fittable}$ (where $\varrho
\in \R,$ $\varrho \neq 0$) if, for all $R > \varrho$, and all
$z\in\D$, there is no other M\"{o}bius transformation $m$ other than
the identity for which $$\int_{\Omega_{z,R}} \,
|\mu(z)-\mu(m(z))|\,d\vol_H(z)=0.$$
\end{defn}

\begin{rem}
This definition can also be read as follows: $\M$ is singly
$\varrho\mbox{-}_{\mbox{\tiny{H}}}\mbox{fittable}$ if and only if,
for all $R>\varrho$, any two conformal factors $\mu_1$ and $\mu_2$
for $\M$ satisfy:
\begin{enumerate}
\item
For all $z\in \D$ there exists a unique minimum to the function $w
\mapsto d^R_{\mu_1,\mu_2}(z,w)$.
\item
For all pairs $(z,w)\in \D \times \D$ that achieve this minimum
there exists a unique M\"{o}bius transformation for which the
integral in {\rm(\ref{e:d_mu,nu(z,w)_def})} vanishes (with $\mu_1$
in the role of $\mu$, and $\mu_2$ in that of $\nu$).
\end{enumerate}
\end{rem}
Essentially, this definition requires that, from some sufficiently
large (hyperbolic) scale onwards, there are no isometric pieces
within $(\D,\mu)$ (or $(\D,\nu)$).


\begin{figure}[t]
\hspace*{.3 in}
\begin{minipage}{2.8 in}
  \caption{Illustration of the proof of Theorem \ref{thm:identity_of_indiscernibles}}\label{fig:for_proof_identity_of_indiscernibles}
\end{minipage}
\begin{minipage}{3 in}
\hspace*{.7 in}
\includegraphics[width=0.8\columnwidth]{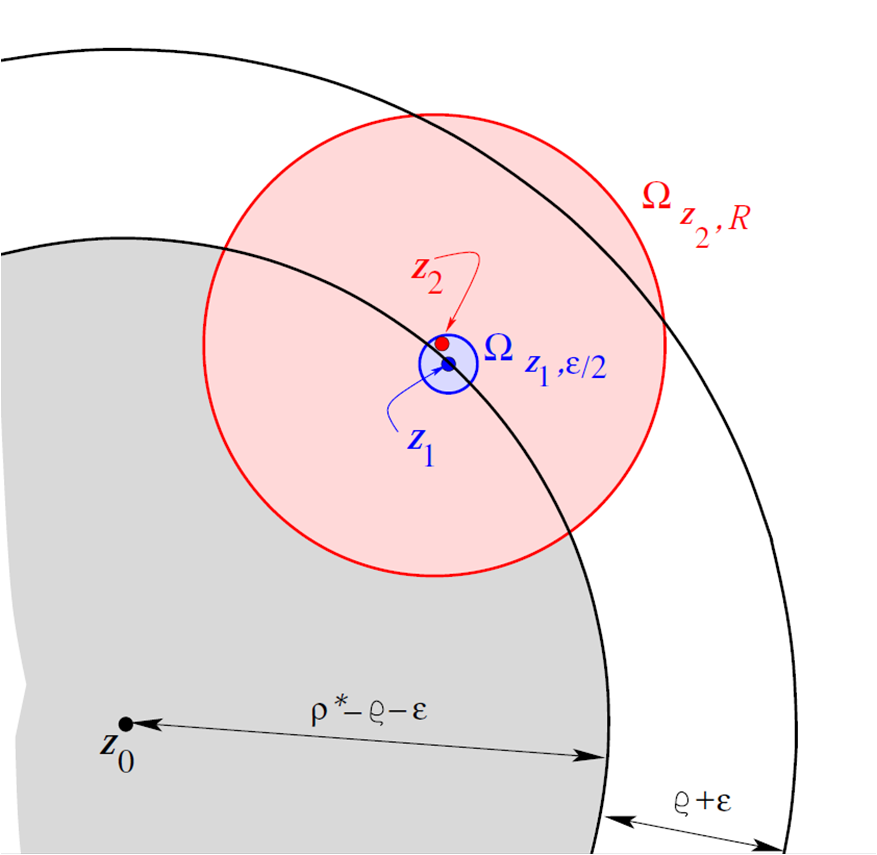}
\end{minipage}
\
\end{figure}

We start with a lemma, and then prove the main result of this
subsection.
\begin{lem}\label{lem:in_every_disk_z_w_d(z,w)=0}
Let $\pi \in \Pi(\mu,\nu)$ be such that $\int_{\D \times
\D}\,d^R_{\mu,\nu}(z,w)\,d\pi(z,w)=0$. Then, for all $z_0 \in \D$
and $\delta >0$, there exists at least one point $z \in
\Omega_{z_0,\delta}$ such that $d^R_{\mu,\nu}(z,w)=0$ for some $w
\in \D$.
\end{lem}
\begin{proof}
By contradiction: assume that there exists a disk
$\Omega_{z_0,\delta}$ such that $d^R_{\mu,\nu}(z,w) >0$ for all
$z\in \Omega_{z_0,\delta}$ and all $w \in \D$. Since
$$
\int_{\Omega(z_0,\delta) \times \D}\,d\pi(z,w) =
\int_{\Omega(z_0,\delta)}\mu(z)\, d\vol_H(z)>0~,
$$
the set $\Omega(z_0,\delta) \times \D$ contains some of the support
of $\pi$. It follows that
$$
\int_{\Omega(z_0,\delta)\times \D} d^R_{\mu,\nu}(z,w) d\pi(z,w)>0~,
$$
which contradicts
$$
 \int_{\Omega(z_0,\delta) \times \D}d^R_{\mu,\nu}(z,w)d\pi(z,w)
\le \int_{\D \times \D}d^R_{\mu,\nu}(z,w)d\pi(z,w) = 0~.
$$
\end{proof}
\begin{thm}\label{thm:identity_of_indiscernibles}
Suppose that $\M$ and $\N$ are two surfaces that are singly
$\varrho\mbox{-}_{\mbox{\tiny{H}}}$fittable. If $\d^R(\M,\N)=0$ for
some $R > \varrho$, then there exists a \Mbs transformation $m \in
\Md$ that is a global isometry between $\M=(\D,\mu)$ and
$\N=(\D,\nu)$ (where $\mu$ and $\nu$ are conformal factors of $\M$
and $\N$, respectively).
\end{thm}
\begin{proof}
When $\d^R(\M,\N)=0$, there exists (see \cite{Villani:2003}) $\pi
\in \Pi(\mu,\nu)$ such that
$$
\int_{\D \times \D}d^R_{\mu,\nu}(z,w)d\pi(z,w)=0.
$$

Next, pick an arbitrary point $z_0 \in \D$ such that, for some $w_0
\in \D$, we have $d^R_{\mu,\nu}(z_0,w_0)=0$. (The existence of such
a pair is guaranteed by Lemma \ref{lem:in_every_disk_z_w_d(z,w)=0}.)
This implies that there exists a unique M\"{o}bius transformation
$m_0 \in \Md$ that takes $z_0$ to $w_0$ and that satisfies
$\nu(m_0(z))=\mu(z)$ for all $z \in \Omega_{z_0,R}$. We define
$$
\rho^* = \sup \{\rho \,;\, d^\rho_{\mu,\nu}(z_0,w_0)=0 \};
$$
clearly $\rho^* \geq R$. The theorem will be proved if we show that
$\rho^*= \infty$. We shall do this by contradiction, i.e. we assume
$\rho^* < \infty$, and then derive a contradiction.

So let's assume $\rho^* < \infty$. Consider $\Omega_{z_0,\rho^*}$,
the hyperbolic disk around $z_0$ of radius $\rho^*$. (See Figure
\ref{fig:for_proof_identity_of_indiscernibles} for illustration.)
Set $\eps = (R - \varrho)/2$, and consider the points on the
hyperbolic circle $C=\partial \Omega_{z_0, \rho^*-\varrho-\eps}$.
For every $z_1 \in C$, consider the hyperbolic disk
$\Omega_{z_1,\eps/2}$; by  Lemma
\ref{lem:in_every_disk_z_w_d(z,w)=0} there exists a point $z_2$ in
this disk and a corresponding point $w_2 \in \D$ such that
$d^R_{\mu,\nu}(z_2,w_2)=0$, i.e. such that
\[
\int_{\Omega_{z_2,R}}\,|\mu(z)-m'^*\nu(z)|^2\,d\vol_H(z) \,=\,0~
\]
for some \Mbs transformation $m'$ that maps $z_2$ to $w_2$; in
particular, we have that
\begin{equation}
\mu(z)=\nu(m'(z))~~\mbox{ for all }~ z \in \Omega_{z_2,R}~.
\label{e:ingrid:mprime}
\end{equation}
The hyperbolic distance from $z_2$ to $\partial \Omega_{z_0,\rho^*}$
is at least $\varrho+\eps/2$.
It follows that the hyperbolic disk $\Omega_{z_2,\varrho+\eps/4}$ is
completely contained in $\Omega_{z_0,\rho^*}$; since
$\mu(z)=\nu(m_0(z))$ for all $z \in \Omega_{z_0,\rho^*}$, this must
therefore hold, in particular, for all $z \in
\Omega_{z_2,\varrho+\eps/4}$. Since
$\Omega_{z_2,\varrho+\eps/4}\subset \Omega_{z_2,R}$, we also have
$\mu(z)=\nu(m'(z))$ for all $z \in \Omega_{z_2,\varrho+\eps/4}$, by
(\ref{e:ingrid:mprime}). This implies $\nu(w)=\nu(m_0\circ
(m')^{-1}(w))$ for all $w \in \Omega_{w_2,\varrho+\eps/4}$. Because
$\N$ is singly $\varrho\mbox{-}_{\mbox{\tiny{H}}}$fittable, it
follows that $m_0\circ (m')^{-1}$ must be the identity, or $m_0=
m'$. Combining this with (\ref{e:ingrid:mprime}), we have thus shown
that $\mu(z)=\nu(m_0(z))$ for all $z \in \Omega_{z_2,R}$.

Since the distance between $z_2$ and $z_1$ is at most $\eps/2$, we
also have
\begin{equation*}
\Omega_{z_2,R} \supset \Omega_{z_1,R - \eps/2} = \Omega_{z_1,\varrho
+ 3\eps/2}~. \label{e:ingrid:supset}
\end{equation*}

This implies that if we select such a point $z_2(z_1)$ for each
$z_1\in C$, then $\Omega_{z_0, \rho^* - \varrho-\eps}\cup
\left(\,\cup_{z_1 \in C}\,\Omega_{z_2(z_1),R}\right)$ covers the
open disk $\Omega_{z_0,\rho^*+\eps/2}$. By our earlier argument,
$\mu(z)=\nu(m_0(z))$ for all $z$ in each of the
$\Omega_{z_2(z_1),R}$; since the same is true on $\Omega_{z_0,
\rho^* - \varrho-\eps}$, it follows that $\mu(z)=\nu(m_0(z))$ for
all $z$ in $\Omega_{z_0,\rho^*+\eps/2}$. This contradicts the
definition of $\rho^*$ as the supremum of all radii for which this
was true; it follows that our initial assumption, that $\rho^*$ is
finite, cannot be true, completing the proof.
\end{proof}

For $(\D,\mu)$ to be singly
$\varrho\mbox{-}_{\mbox{\tiny{H}}}$fittable, no two hyperbolic disks
$\Omega_{z,R}$, $\Omega_{w,R}$ (where $w$ can equal $z$) can be
isometric via a M\"{o}bius transformation $m$, if $R > \varrho$,
except if $m=Id$. However, if $z$ is close (in the Euclidean sense)
to the boundary of $\D$, the hyperbolic disk $\Omega_{z,R}$ is very
small in the Euclidean sense, and corresponds to a very small piece
(near the boundary) of $\M$. This means that single
$\varrho\mbox{-}_{\mbox{\tiny{H}}}$fittability imposes restrictions
in increasingly small scales near the boundary of $\M$; from a
practical point of view, this is hard to check, and in many
applications, the behavior of $\M$ close to its boundary is
irrelevant. For this reason, we also formulate the following
relaxation of the results above.

\begin{defn}
A surface $\M$ is said to be a singly $A\mbox{-}_{\!_{\M}}$fittable
(where $A> 0$) if there are no patches (i.e. open, path-connected
sets) in $\M$ of area larger than $A$ that are isometric, with
respect to the metric on $\M$.
\end{defn}

If a surface is singly $A\mbox{-}_{\!_{\M}}$fittable, then it is
obviously also $A'\mbox{-}_{\!_{\M}}$fittable for all $A' \geq A$;
the condition of being $A\mbox{-}_{\!_{\M}}$fittable becomes more
restrictive as $A$ decreases. The following theorem states that two
singly $A\mbox{-}_{\!_{\M}}$fittable surfaces at zero
$\d^R$-distance from each other must necessarily be isometric, up to
some small boundary layer.

\begin{thm}
Consider two surfaces $\M$ and $\N$, with corresponding conformal
factors $\mu$ and $\nu$ on $\D$, and suppose $\d^R(\M,\N)=0$ for
some $R>0$. Then the following holds: for arbitrarily large
$\rho>0$, there exist a \Mbs transformation $m \in M_D$ and a value
$A>0$ such that if $\M$ and $\N$ are singly
$A\mbox{-}_{\!_{\M}}$fittable then $\mu(m(z))=\nu(z)$, for all $z
\in \Omega_{0,\rho}$.
\end{thm}
\begin{proof}
Part of the proof follows the same lines as for Theorem
\ref{thm:identity_of_indiscernibles}. We highlight here only the new
elements needed for this proof.

First, note that, for arbitrary $r>0$ and $z_0 \in \D$,
\begin{equation}\label{e:lower_bound_for_patch_area}
    \Vol_\M(\Omega_{z_0,r}) = \int_{\Omega_{z_0,r}}\mu(z)d\vol_H(z)
    \geq \Vol_H(\Omega_{z_0,r})\left[\min_{z\in\Omega_{z_0,r}}\mu(z)\right] =
    \Vol_H(\Omega_{0,r})\left[\min_{z\in\Omega_{z_0,r}}\mu(z)\right].
\end{equation}
This motivates the definition of the sets $\OO_{A,r}$,
\begin{equation}\label{e:O_big_area_set}
    \OO_{A,r} = \set{ z\in \D \ \mid \ \min_{z'\in\Omega_{z,r}}\mu(z') >
    \frac{A}{\Vol_H(0,\Omega_{0,r})}
    };
\end{equation}
$A>0$ is still arbitrary at this point; its value will be set below.

Now pick $r<R$, and set $\eps=(R-r)/2$. Note that if $z \in
\OO_{A,r}$, then
$\Vol_{\M}(\Omega_{z,R})\geq\Vol_{\M}(\Omega_{z,r})>A $.

Since $\mu$ is bounded below by a strictly positive constant on each
$\Omega_{0,\rho'}$, we can pick, for arbitrarily large $\rho$, $A>0$
such that $\Omega_{0,\rho} \subset \OO_{A,r}$; for this it suffices
that $A$ exceed a threshold depending on $\rho$ and $r$. (Since
$\mu(z)\rightarrow 0$ as $z$ approaches the boundary of $\D$ in
Euclidean norm, we expect this threshold to tend towards $0$ as
$\rho \rightarrow \infty$.) We assume that $\Omega_{0,\rho} \subset
\OO_{A,r}$ in what follows.

Similar to the proof of Theorem
\ref{thm:identity_of_indiscernibles}, we invoke Lemma
\ref{lem:in_every_disk_z_w_d(z,w)=0} to infer the existence of
$z_0,w_0$ such that $z_0\in\Omega_{0,\eps/2}$  and
$d^R_{\mu,\nu}(z_0,w_0)=0$. We denote
$$
\rho^* = \sup \{r' \,;\, d^{r'}_{\mu,\nu}(z_0,w_0)=0 \};
$$
as before, there exists a \Mbs transformation $m$ such that
$\nu(m(z))=\mu(z)$ for all $z$ in $\Omega_{z_0,\rho^*}$. To complete
our proof it therefore suffices to show that $\rho^* \geq \rho +
\eps/2$, since $\Omega_{0,\rho}\subset \Omega_{z_0,\rho+\eps/2}$ .

Suppose the opposite is true, i.e. $\rho^* < \rho+\eps/2$. By the
same arguments as in the proof of Theorem
\ref{thm:identity_of_indiscernibles},  there exists, for each
$z_1\in \partial \Omega_{z_0,\rho^*-r-\eps}$, a point $z_2 \in
\Omega_{z_1,\eps/2}$ such that $d^R_{\mu,\nu}(z_2,w_2)=0$ for some
$w_2$. Since the hyperbolic distance between $z_2$ and $0$ is
bounded above by $\eps/2+\rho^*-r-\eps+\eps/2<\rho-r+\eps/2<\rho$,
$z_2 \in \Omega_{0,\rho} \subset \OO_{A,r}$, so that
$\Vol_{\M}(\Omega_{z_2,R})>A $. It then follows from the conditions
on $\M$ and $\N$ that $\nu(m(z))=\mu(z)$ for all z in
$\Omega_{z_0,\rho^*} \cup \Omega_{z_2,R} \supset
\Omega_{z_0,\rho^*}\cup \Omega_{z_1,r+3\eps/2}$. Repeating the
argument for all $z_1\in \partial \Omega_{z_0,\rho^*-r-\eps}$ shows
that $\nu(m(z))=\mu(z)$ can be extended to all $z \in
\Omega_{z_0,\rho^*+\eps/2}$, leading to a contradiction that
completes the proof.
\end{proof}

\section{Discretization and implementation}
\label{s:the_discrete_case_implementation} To transform the
theoretical framework constructed in the preceding sections into an
algorithm, we need to discretize the relevant continuous objects.
Our general plan is to recast the transportation
eq.~(\ref{e:generalized_Kantorovich_transportation}) as a linear
programming problem between discrete measures. This requires two approximation steps: \\
1) approximating the surface's Uniformization, and \\
2) discretizing the resulting continuous measures and finding the
optimal transport between the discrete measures.

To show how we do this, we first review a few basic notions such as
the representation of (approximations to) surfaces by faceted,
piecewise flat approximations, called {\em meshes}, and discrete
conformal mappings; the conventions we describe here are the same as
adopted in \cite{Lipman:2009:MVF}.

\subsection{Meshes, mid-edge meshes, and discrete conformal mapping}
Triangular (piecewise-linear) meshes are a popular choice for the
definition of discrete versions of smooth surfaces. We shall denote
a triangular mesh by the triple $M = (\V, \E, \F)$, where
$\V=\{v_i\}_{i=1}^m \subset \Real^3$ is the set of vertices,
$\E=\{e_{i,j}\}$ the set of edges, and $\F=\{f_{i,j,k}\}$  the set
of faces (oriented $i\rightarrow j \rightarrow k$). When dealing
with a second surface, we shall denote its mesh by $N$. We assume
our mesh is homeomorphic to a disk.

Next, we introduce ``conformal mappings'' of a mesh to the unit
disk. Natural candidates for discrete conformal mappings are not
immediately obvious. Since we are dealing with piecewise linear
surfaces, it might seem natural to select a continuous linear maps
that is piecewise affine, such that its restriction to each triangle
is a similarity transformation. A priori, a similarity map from a
triangular face to the disk has 4 degrees of freedom; requiring that
the image of each edge remain a shared part of the boundary of the
images of the faces abutting the edge, and that the map be
continuous when crossing this boundary, imposes 4 constraints for
each edge. This quick back of the envelope calculation thus allows
$4|\F|$ degrees of freedom for such a construction, with $4|\E|$
constraints. Since $3|\F|/2 \approx |\E|$ this problem is over
constrained, and a construction along these lines is not possible. A
different approach uses the notion of discrete harmonic and discrete
conjugate harmonic functions due to Pinkall and Polthier
\cite{Pinkall93,Polthier05} to define a discrete conformal mapping
on the mid-edge mesh (to be defined shortly). This relaxes the
problem to define a map via a similarity on each triangle that is
continuous through only \emph{one} point in each edge, namely the
mid point. This procedure was employed in \cite{Lipman:2009:MVF}; we
will summarize it here; for additional implementation details we
refer the interested reader (or programmer) to that paper, which
includes a pseudo-code.

The mid-edge mesh $\mM = (\mV, \mE, \mF)$ of a given mesh $M=(\V,
\E, \F)$ is defined as follows. For the vertices $\mv_r \in \mV$, we
pick the mid-points of the edges of the mesh $M$; we call these the
mid-edge points of $M$. There is thus a $\mv_r \in \mV$
corresponding to each edge $e_{i,j} \in \E$. If $\mv_s$ and $\mv_r$
are the mid-points of edges in $\E$ that share a vertex in $M$, then
there is an edge $\me_{s,r} \in \mE$ that connects them. It follows
that for each face $f_{i,j,k} \in \F$ we can define a corresponding
face $\mf_{r,s,t} \in \mF$, the vertices of which are the mid-edge
points of (the edges of) $f_{i,j,k}$; this face has the same
orientation as $f_{i,j,k}$.  Note that the mid-edge mesh is not a
manifold mesh, as illustrated by the mid-edge mesh in Figure
\ref{f:discrete_type_1}, shown together with its ``parent'' mesh: in
$\mM$ each edge ``belongs'' to only one face $\mF$, as opposed to a
manifold mesh, in which most edges (the edges on the boundary are
exceptions) function as a hinge between two faces. This ``lace''
structure makes a mid-edge mesh more flexible: it turns out that it
is possible to define a piecewise linear map that makes each face in
$\mF$ undergo a pure scaling (i.e. all its edges are shrunk or
extended by the same factor) and that simultaneously flattens the
whole mid-edge mesh. By extending this back to the original mesh, we
thus obtain a map from each triangular face to a similar triangle in
the plane; these individual similarities can be ``knitted together''
through the mid-edge points, which continue to coincide (unlike most
of the vertices of the original triangles).

\begin{figure}[ht]
\centering \setlength{\tabcolsep}{0.4cm}
\begin{tabular}{@{\hspace{0.0cm}}c@{\hspace{0.2cm}}c@{\hspace{0.0cm}}}
\includegraphics[width=0.4\columnwidth]{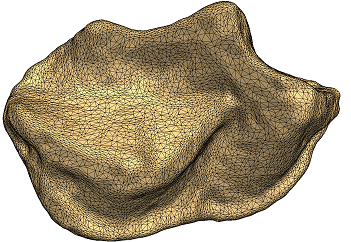}
&
\includegraphics[width=0.4\columnwidth]{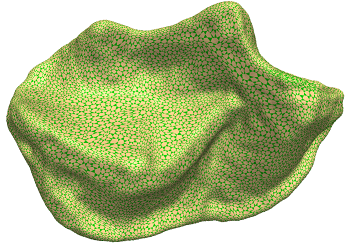} \\
Discrete mesh& Mid-edge mesh\\
\includegraphics[width=0.4\columnwidth]{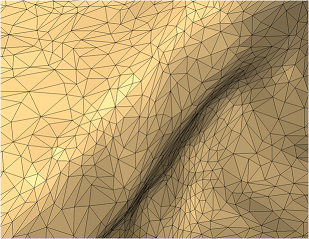}
&
\includegraphics[width=0.4\columnwidth]{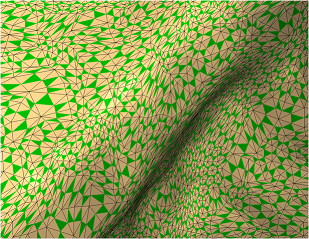} \\
Surface mesh zoom-in & Mid-edge mesh zoom-in
\end{tabular}
\caption{A mammalian tooth surface mesh, with the corresponding
mid-edge mesh.}\label{f:discrete_type_1}
\end{figure}

To determine the flattening map, we use the framework of discrete
harmonic and conjugate harmonic functions, first defined and studied
by Pinkall and Polthier \cite{Pinkall93,Polthier05} in the context
of discrete minimal surfaces. This framework was first adapted to
the present context in \cite{Lipman:2009:MVF}; this adaptation is
explained in some detail in Appendix B. The flattening map is
well-defined at the mid-edges $\mv_s$.
As shown in \cite{Lipman:2009:MVF} (see also Appendix B) the
boundary of the mesh gets mapped onto a region with a straight
horizontal slit (see Figure \ref{f:discrete_type_2}, where the
boundary points are marked in red). We can assume, without loss of
generality, that this slit coincides with the interval $[-2, 2]
\subset \C$, since it would suffice to shift and scale the whole
figure to make this happen. The holomorphic map $z=w+\frac{1}{w}$
maps the unit disk conformally to $\mathbb{C}\setminus [-2,2]$, with
the boundary of the disk mapped to the slit at $[-2,2]$; when the
inverse of this map is applied to our flattened mid-edge mesh, its
image will thus be a mid-edge mesh in the unit disk, with the
boundary of the disk corresponding to the boundary of our
(disk-like) surface. (See Figure \ref{f:discrete_type_2}.) We shall
denote by $\Phi:\mV \rightarrow \C$ the concatenation of these
different conformal and discrete-conformal maps, from the original
mid-edge mesh to the corresponding mid-edge mesh in the unit disk.

\begin{figure}[ht]
\centering \setlength{\tabcolsep}{0.4cm}
\begin{tabular}{@{\hspace{0.0cm}}c@{\hspace{0.0cm}}c@{\hspace{0.0cm}}c@{\hspace{0.0cm}}}
\includegraphics[width=0.2\columnwidth]{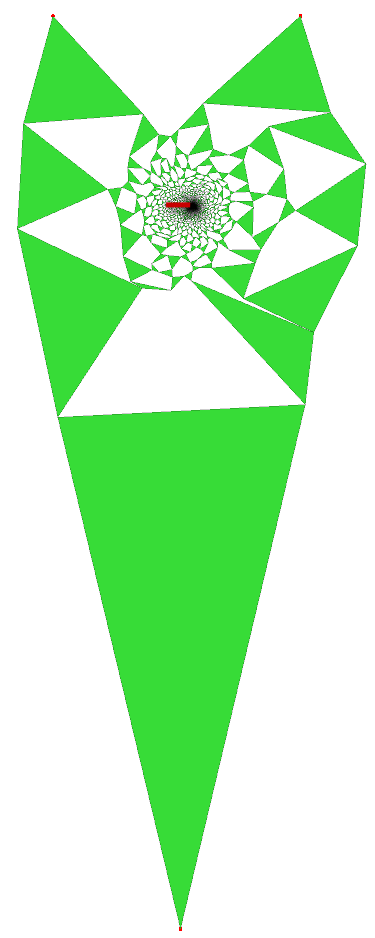} &
\includegraphics[width=0.4\columnwidth]{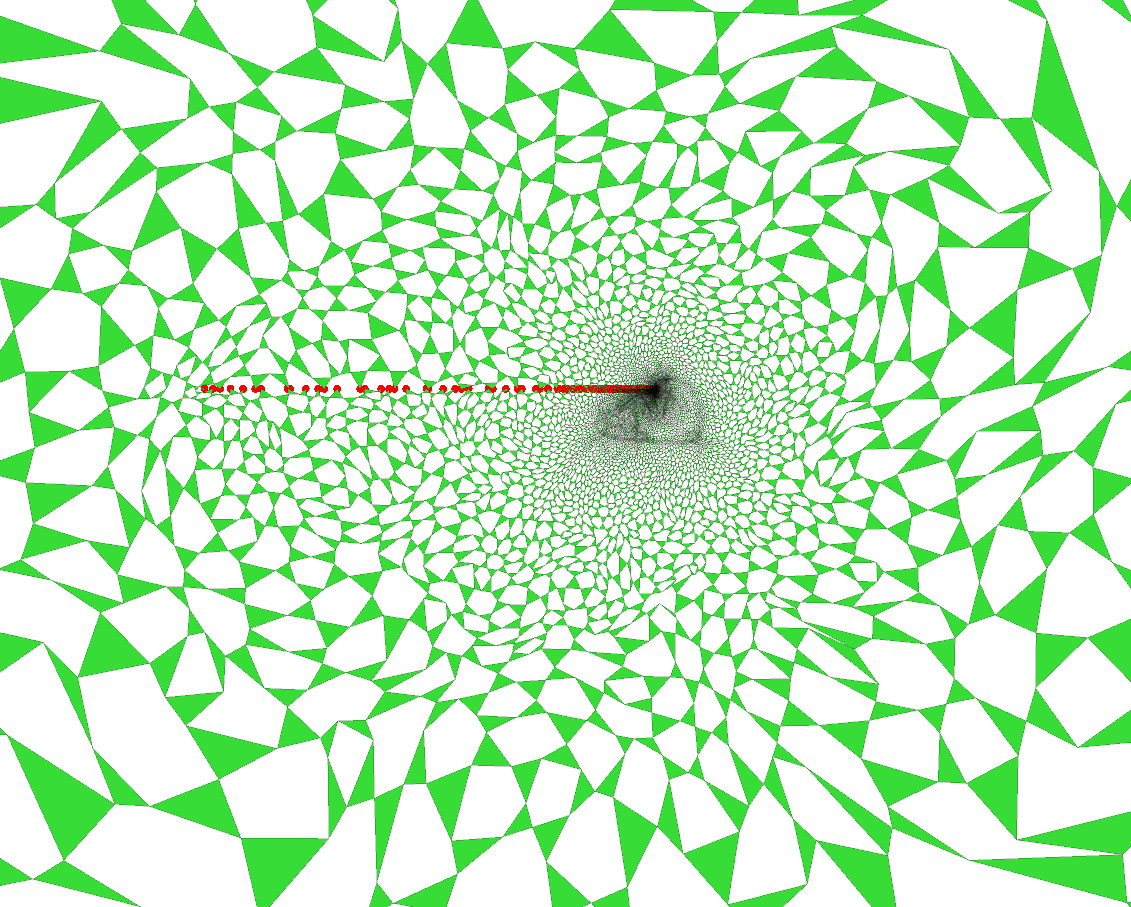}\\
Mid-edge uniformization & Uniformization Zoom-in \\
\includegraphics[width=0.4\columnwidth]{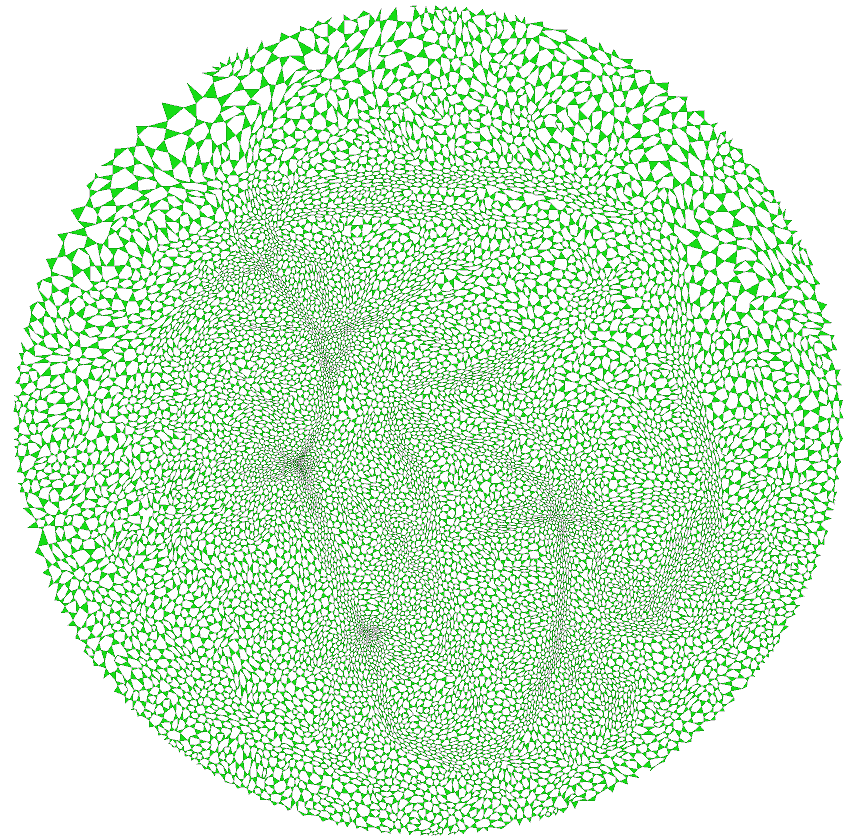}& \includegraphics[width=0.5\columnwidth]{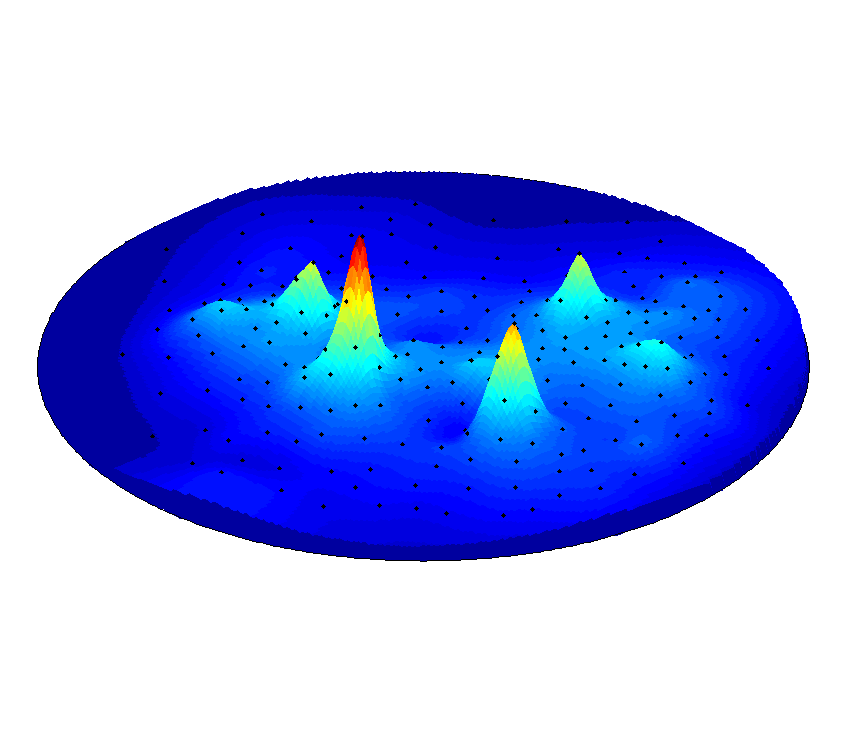}\\
After mapping to the disk & Interpolated conformal factor
\end{tabular}\\
\caption{The discrete conformal transform to the unit disk for the
surface of Figure \ref{f:discrete_type_1}, and the interpolation of
the corresponding discrete conformal factors (plotted with the JET
color map in Matlab). The red points in the top row's images show
the boundary points of the disk.} \label{f:discrete_type_2}
\end{figure}

Next, we define the Euclidean discrete conformal factors, defined as
the density, w.r.t. the Euclidean metric, of the mid-edge triangles
(faces), i.e.
\[
\mu^E_{\mf_{r,s,t}} =
\frac{\vol_{\mathds{R}^3}(\mf_{r,s,t})}{\vol(\Phi(\mf_{r,s,t}))}.
\]
Note that according to this definition, we have
\[
\int_{\Phi(\mf_{r,s,t})}\,\mu^E_{\mf_{r,s,t}}\,d\vol_E\, =
\,\frac{\vol_{\mathds{R}^3}(\mf_{r,s,t})}{\vol_E\left(\Phi(\mf_{r,s,t})\right)}\,
\vol_E\left(\Phi(\mf_{r,s,t})\right)\,=\,\vol_{\mathds{R}^3}(\mf_{r,s,t}),
\]
where $\vol_E$ denotes the standard Euclidean volume element
$dx^1\wedge dx^2$ in $\C$, and $\vol_{\R^3}(\mf)$ stands for the
area of $\mf$ as induced by the standard Euclidean volume element in
$\R^3$. The discrete Euclidean conformal factor at a mid-edge vertex
$\mv_r$ is then defined as the average of the conformal factors for
the two faces $\mf_{r,s,t}$ and $\mf_{r,s',t'}$ that touch in
$\mv_r$, i.e.
\[
\mu^E_{\mv_r} \,=\,
\frac{1}{2}\,\left(\mu^E_{\mf_{r,s,t}}\,+\,\mu^E_{\mf_{r,s',t'}}\right).
\]
Figure \ref{f:discrete_type_2} illustrates the values of the
Euclidean conformal factor for the mammalian tooth surface of
earlier figures. The discrete hyperbolic conformal factors are
defined according to the following equation, consistent with the
convention adopted in section \ref{s:prelim},
\begin{equation}\label{e:discrete_hyperbolic_conformal_density}
    \mu^H_{\mv_r} \,= \,\mu^E_{\mv_r} \,\left(1 - |\Phi(\mv_r)|^2\right)^2.
\end{equation}
As before, we shall often drop the superscript: unless otherwise
stated, $\mu=\mu^H$, and $\nu=\nu^H$.

The (approximately) conformal mapping of the original mesh to the
disk is completed by constructing a smooth interpolant $\Gamma_\mu
:\D \rightarrow \R$, which interpolates the discrete conformal
factor so far defined only at the vertices in $\Phi(\mV)$;
$\Gamma_\nu$ is constructed in the same way. In practice we use
Thin-Plate Splines, i.e. functions of the type
\[
\Gamma_\mu(z) = p_1(z)\, +\,\sum_i\, b_i \,\psi(|z-z_i|)\,,
\]
where $\psi(r)=r^2\log(r^2)$, $p_1(z)$ is a linear polynomial in
$x^1,x^2$, and  $b_i \in \C$; $p_1$ and the $b_i$ are determined by
the data that need to be interpolated. Similarly
$\Gamma_\nu(w)\,=\,q_1(w)\,+\,\sum_j \,c_j\, \psi(|w-w_j|)$ for some
constants $c_j \in \C$ and a linear polynomial $q_1(w)$ in
$y^1,y^2$. We use as interpolation centers two point sets
$Z=\set{z_i}_{i=1}^n,$ and $W=\set{w_j}_{j=1}^p$ defined in the next
subsection for the discretization of measures. See Figure
\ref{f:discrete_type_2} (bottom-right) the interpolated conformal
factor based on the black point set.

We also note that for practical purposes it is sometimes
advantageous to use Smoothing Thin-Plate Splines:
$$\Gamma_\mu(z) = \mathop{\mathrm{argmin}}_{\gamma} \set{ \lambda \sum_{r} \abs{\mu_{\mv_r} - \gamma(\Phi(\mv_r))}^2 + (1-\lambda) \int_{\D} \parr{\frac{\partial^2 \gamma }{\partial (x^1)^2}}^2 + \parr{\frac{\partial^2 \gamma }{\partial x^1 x^2}}^2 + \parr{\frac{\partial^2 \gamma }{\partial (x^2)^2}}^2  dx^1 \wedge dx^2 }.$$
when using these, we picked the value $0.99$ for the smoothing
factor $\lambda$.

\subsection{Discretizing continuous measures and their transport}
\label{ss:discretizing_discrete_measures}

In this subsection we indicate how to construct discrete
approximations ${T_{\mbox{\footnotesize{discr.}},d}^R}(\xi,\zeta)$
for the distance $\d^R(\X,\Y)=T^R_d(\xi,\zeta)$ between two surfaces
$\X$ and $\Y$, each characterized by a corresponding smooth density
on the unit disk $\D$ ($\xi$ for $\X$, $\zeta$ for $\Y$). (In
practice, we will use the smooth functions $\Gamma_{\mu}$ and
$\Gamma_{\nu}$ for $\xi$ and $\zeta$. ) We shall use discrete
optimal transport to construct our approximation
${T_{\mbox{\footnotesize{discr.}},d}^R}(\xi,\zeta)$, based on
sampling sets for the surfaces, with convergence to the continuous
distance as the sampling is refined.

To quantify how fine a sampling set $Z$ is, we use the notion of
\emph{fill distance} $\varphi(Z)$:
\[
\varphi(Z)\,:=\,\sup \set{r>0\ \big| \ z \in \M:B_g(z,r)\cap Z_h =
\emptyset}~,
\]
where $B_g(z,r)$ is the geodesic open ball of radius $r$ centered at
$z$. That is, $\varphi(Z)$ is the radius of the largest geodesic
ball that can be fitted on the surface $\X$ without including any
point of $Z$. The smaller $\varphi(Z)$, the finer the sampling set.

Given the smooth density $\xi$ (on $\D$), we discretize it by first
distributing $n$ points $Z=\{z_i\}_{i=1}^n$ on $\X$ with
$\varphi(Z)=h>0$. For $i=1,\ldots,n$, we define the sets $\Xi_i$ to
be the Voronoi cells corresponding to $z_i \in Z$; this gives a
partition of the surface $\X$ into disjoint convex sets,
$\X=\cup_{i=1}^n \Xi_i$. We next define the discrete measure $\xi_Z$
as a superposition of point measures localized in the points of $Z$,
with weights given by the areas of $\Xi_i$, i.e. $\xi_Z
\,=\,\sum_{i=1}^n\,\xi_i \delta_{z_i}$, with
$\xi_i:=\xi(\Xi_i)=\int_{\Xi_i}d\vol_\X$. Similarly we denote by
$W=\{w_j\}_{j=1}^p$, $\Upsilon_j$, and $\zeta_j:=\zeta(\Upsilon_j)$
the corresponding quantities for surface $\Y$. We shall always
assume that the surfaces $\X$ and $\Y$ have the same area, which,
for convenience, we can take to be 1. It then follows that the
discrete measures $\xi_Z$ and $\zeta_W$ have equal total mass
(regardless of whether $n=p$ or not). The approximation algorithm
will compute optimal transport for the discrete measures $\xi_Z$ and
$\zeta_W$; the corresponding discrete approximation to the distance
between $\xi$ and $\zeta$ is then given by $T^R_d(\xi_Z,\zeta_W)$.

Convergence of the discrete approximations $T^R_d(\xi_Z,\zeta_W)$ to
$T^R_d(\xi,\zeta)=\d^R(\X,\Y)$ as $\varphi(Z)$, $\varphi(W) \too 0$
then follows from the results proved in
\cite{Lipman:TR2009:approxOptimal}. Corollary 3.3 in
\cite{Lipman:TR2009:approxOptimal} requires that the distance
function $d^R_{\xi,\zeta}(\cdot,\cdot)$ used to define
$T^R_d(\xi,\zeta)$ be uniformly continuous in its two arguments. We
can establish this in our present case by invoking the continuity
properties of $d^R_{\xi,\zeta}$ proved in Theorem
\ref{thm:continuity_of_d^R_mu,nu}, extended by the following lemma,
proved in Appendix A.
\begin{lem}\label{lem:extension_of_d}
Let $\set{(z_k,w_k)}_{k\geq 1} \subset \D\times\D$ be a sequence
that converges, in the Euclidean norm, to some point in $(z',w') \in
\bbar{\D}\times\bbar{\D} \setminus \D\times\D$, that is
$|z_k-z'|+|w_k-w'| \too 0$, as $k \too \infty$. Then, $\lim_{k\too
\infty}d^R_{\xi,\zeta}(z_k,w_k)$ exists and depends only on the
limit point $(z',w')$.
\end{lem}

We shall denote this continuous extension of
$d^R_{\mu,\nu}(\cdot,\cdot)$ to $\bbar{\D}\times \bbar{\D}$ by the
same symbol $d^R_{\mu,\nu}$.

Since $\bbar{\D} \times \bbar{D}$ is compact, (this extension of)
$d^R_{\xi,\zeta}(\cdot,\cdot)$ is uniformly continuous: for all
$\eps >0$, there exists a $\delta=\delta(\eps)$ such that, for all
$z,z' \in \X$, $w,w' \in \Y$,
\[
d_\X(z,z') < \delta(\eps) \ , d_\Y(w,w') < \delta(\eps) \
\Rightarrow \ \abs{d^R_{\xi,\zeta}(z,w) - d^R_{\xi,\zeta}(z',w') } <
\eps,
\]
where $d_\X(\cdot,\cdot)$ is the geodesic distance on $\X$, and
$d_\Y(\cdot,\cdot)$ is the geodesic distance on $\Y$.

The results in \cite{Lipman:TR2009:approxOptimal} then imply that
$\xi_Z \too \xi$ in the \emph{weak} sense, as $\varphi(Z) \too 0$,
i.e. that for all bounded continuous functions $f:\bbar{\D}\too
\Real$, the convergence $\int_\D f\ d\xi_Z \too \int_\D f\ d\xi$
holds \cite{Billingsley68}. Similarly and $\zeta_W \too \zeta$ in
the weak sense as $\varphi(W) \too 0$. Furthermore,
\cite{Lipman:TR2009:approxOptimal} also proves that for
$\max(\varphi(Z),\varphi(W)) < \frac{\delta(\eps)}{2}$
\[
\abs{T^R_d(\xi_Z,\zeta_W) - T^R_d(\xi,\zeta)} < \eps.
\]
More generally, it is shown that
\begin{equation}\label{e:optimal_trans_discrete_approx_with_mod_cont}
    \abs{T^R_d(\xi_Z,\zeta_W) - T^R_d(\xi,\zeta)} <
\omega_{d^R_{\xi,\zeta}}\parr{\max(\varphi(Z),\varphi(W))},
\end{equation}
where $\omega_{d^R_{\xi,\zeta}}$ is the modulus of continuity of
$d^R_{\xi,\zeta}$, that is
$$\omega_{d^R_{\xi,\zeta}}(t)=\sup_{d_\X(z,z') + d_\Y(w,w') <
t}\abs{d^R_{\xi,\zeta}(z,w)-d^R_{\xi,\zeta}(z',w')}.$$

We shall see below that it will be particularly useful to choose the
centers in $Z=\set{z_i}_{i=1}^n$, $W=\set{w_j}_{j=1}^p$ such that
the corresponding Voronoi cells are (approximately) of equal area,
i.e. $n=N=p$ and $\xi_i=\xi(\Xi_i)\approx\frac{1}{N}$,
$\zeta_j=\zeta(\Upsilon_j)\approx\frac{1}{N}$, where we have used
that the total area of each surface is normalized to $1$. An
effective way to calculate such sample sets $Z$ and $W$ is to start
from an initial random seed (which will not be included in the set),
and take the geodesic point furthest from the seed as the initial
point of the sample set. One then keeps repeating this procedure,
selecting at each iteration the point that lies at the furthest
geodesic distance from the set of points already selected. This
algorithm is known as the Farthest Point Algorithm (FPS)
\cite{eldar97farthest}. An example of the output of this algorithm,
using geodesic distances on a disk-type surface, is shown in Figure
\ref{f:discrete_type_3}. Further discussion of practical aspects of
Voronoi sampling of a surface can be found in
\cite{BBK2007non_rigid_book}.

\begin{figure}[h]
\centering \setlength{\tabcolsep}{0.4cm} \hspace*{.5 in}
\begin{minipage}{3 in}
\includegraphics[width=0.8\columnwidth]{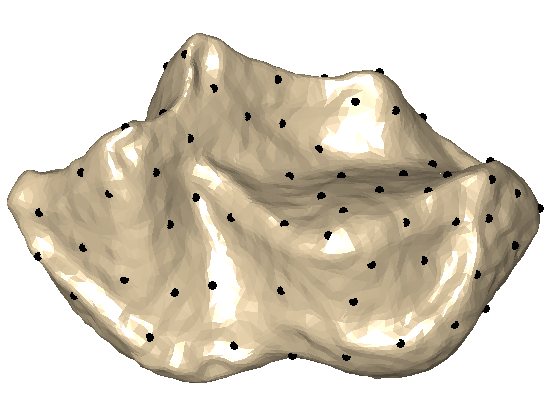}\hspace*{-.5 in}
\end{minipage}
\hspace*{-.5 in}
\begin{minipage}{3 in}
\caption{Sampling of the surface of  Figure \ref{f:discrete_type_1}
obtained by the Farthest Point  Algorithm.}
\label{f:discrete_type_3}
\end{minipage}
\end{figure}

\subsection{Approximating the local distance function $d^R_{\mu,\nu}$.}
We are now ready to construct our discrete version of the optimal
volume transportation for surfaces
(\ref{e:generalized_Kantorovich_transportation}). The previous
subsection describes how to derive the discrete measures
$\mu_Z,\nu_W$ from the approximate conformal densities
$\Gamma_\mu,\Gamma_\nu$ and the sampling sets $Z$ and $W$. For
simplicity, we will, with some abuse of notation, identify the
approximations $\Gamma_\mu,\Gamma_\nu$ with $\mu,\nu$. The
approximation error made here is typically much smaller than the
errors made in further steps (see below) and we shall neglect it.
The final component is approximating $d^R_{\mu,\nu}(z_i,w_j)$ for
all pairs $(z_i,w_j) \in Z\times W$. Applying
(\ref{e:d_mu,nu(z,w)_def}) to the points $z_i$, $w_j$ we have:
\begin{equation}\label{e:d_mu,nu(z_i,w_j)}
    d^R_{\mu,\nu}(z_i,w_j) = \min_{m(z_i)=w_j}\int_{\Omega_{z_i, R}}\Big|\,\mu(z)-\nu(m(z))\,\Big |\,d\vol_H.
\end{equation}

\begin{figure}
  \includegraphics[width=0.7\columnwidth]{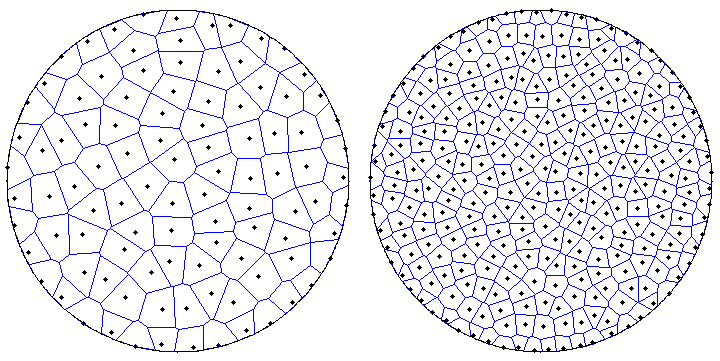}\\
  \caption{The integration centers and their corresponding Voronoi cells used
for calculating the integration weights for the discrete quadrature.
Left: $100$ centers; Right: $300$.}\label{fig:integration_pnts}
\end{figure}

To obtain $d^R_{\mu,\nu}(z_i,w_j)$ we will thus need to compute
integrals over hyperbolic disks of radius $R$, which is done via a
separate approximation procedure, set up once and for all in a
preprocessing step at the start of the algorithm.

By using a \Mbs transformation $\widetilde{m}$ such that
$\widetilde{m}(0)=z_0$, and the identity
\[
\int_{\Omega_{0,R}}\Big|\,\mu(\widetilde{m}(u)) - \nu(m \circ
\widetilde{m} (u))\,\Big| \,d\vol_H(u) =
\int_{\Omega_{z_0,R}}\Big|\,\mu(z)) - \nu(m (z))\,\Big|
\,d\vol_H(z)~,
\]
we can reduce the integrals over the hyperbolic disks
$\Omega_{z_i,R}$ to integrals over a hyperbolic disk centered around
zero.

In order to (approximately) compute integrals over $\Omega_0
=\Omega_{0,R}=\{z |\ |z|\leq r_R\}$, we first pick a positive
integer $K$ and distribute centers $p_k,\,k=1,...,K $ in $\Omega_0$.
We then decompose $\Omega_0$ into Voronoi cells $\Delta_k$
corresponding to the $p_k$, obtaining $\Omega_0 = \cup_{k=1}^K
\Delta_k$; see Figure \ref{fig:integration_pnts} (note that these
Voronoi cells are completely independent of those used in
\ref{ss:discretizing_discrete_measures}.)

To approximate the integral of a continuous function $f$ over
$\Omega_0$ we then use
\[
\int_{\Omega_0}\, f(z)\, d\vol_H(z) \approx \sum_k \,
\brac{\int_{\Delta_k}d\vol_H(z)}\,f(p_k) = \sum_k \,\alpha_k f(p_k)
\]
where $\alpha_k = \int_{\Delta_k}d\vol_H(z)$.

We thus have the following approximation:
\begin{eqnarray}
d^R_{\mu,\nu}(z_i,w_j)&=&
\min_{m(z_i)=w_j}\int_{\Omega_{z_i,R}}\Big|\,\mu(z) - \nu(m
(z))\,\Big|
\,d\vol_H(z)\nonumber\\
&=& \min_{m(z_i)=w_j}\int_{\Omega_{0,R}}\Big|\,\mu(\widetilde{m}_i
(z)) -
\nu(m(\widetilde{m}_i (z)))\,\Big|\,d\vol_H(z)\nonumber\\
&\approx&   \min_{m(z_i)=w_j}  \sum_k \,\alpha_k
\,\Big|\,\mu(\widetilde{m}_i (p_k)) - \nu(m(\widetilde{m}_i
(p_k)))\,\Big|~,\label{e:discrete_approx__d_mu_nu(z_i,w_j)}
\end{eqnarray}
where the \Mbs transformations $\widetilde{m}_i$, mapping 0 to
$z_i$, are selected as soon as the $z_i$ themselves have been
picked, and remain the same throughout the remainder of the
algorithm.

It can be shown that picking a set of centers $\set{p_k}$ with
fill-distance $h>0$ leads to an $O(h)$ approximation;  in
Appendix~\ref{a:appendix A} we prove:
\begin{thm}\label{thm:convergence_of_numerical_quadrature}
For continuously differentiable $\mu,\nu$,
\[
\left|\,d^R_{\mu,\nu}(z_i,w_j)- \min_{m(z_i)=w_j}\sum_k \,\alpha_k
\,\left|\,\mu(\widetilde{m}_i (p_k)) - \nu(m(\widetilde{m}_i
(p_k)))\,\right|\, \right| \leq C\,\varphi\left( \set{p_k} \right)~,
\]
where the constant $C$ depends only on $\mu,\nu,R$.
\end{thm}

Let us denote this approximation by $$\wh{d}^R_{\mu,\nu}(z_i,w_j) =
\min_{m(z_i)=w_j}\sum_k \,\alpha_k \,\left|\,\mu(\widetilde{m}_i
(p_k)) - \nu(m(\widetilde{m}_i (p_k)))\,\right|.$$ Since the above
theorem guarantees that the approximation error
$\abs{\wh{d}^R_{\mu,\nu}(z_i,w_j) - d^R_{\mu,\nu}(z_i,w_j)}$ can be
uniformly bounded independently of $z_i,w_j$, it can be shown that
$$\babs{T^R_d(\mu_Z,\nu_W) - T^R_{\wh{d}}(\mu_Z,\nu_W)}
\leq C \varphi \parr{\set{p_k}},$$ where again $C$ is dependent only
upon $\mu,\nu,R$. Combining this with
eq.(\ref{e:optimal_trans_discrete_approx_with_mod_cont}) we get that
\begin{equation}\label{e:final_approx_error}
    \babs{T^R_d(\mu,\nu) - T^R_{\wh{d}}(\mu_Z,\nu_W)}
\leq \omega_{d^R_{\mu,\nu}}\parr{\max \parr{\varphi(Z),\varphi(W)}}
+ C\varphi  \parr{\set{p_k}}.
\end{equation}

In practice, for calculating $\wh{d}^R_{\mu,\nu}$, the minimization
over $M_{D,z_i,w_j}$, the set of all \Mbs transformations that map
$z_i$ to $w_k$, is discretized as well: instead of minimizing over
all $m_{z_i,w_j,\sigma}$ (see subsection 3.1), we minimize over only
the \Mbs transformations
$\left(m_{z_i,w_j,2\pi\ell/L}\right)_{\ell=0,1,..,L-1}$. Taking this
into account as well, we have thus
\begin{eqnarray}
d^R_{\mu,\nu}(z_i,w_j)\approx \min_{\ell=1,\ldots L}  \sum_k
\,\alpha_k \,\Big|\,\mu(\widetilde{m}_i (p_k)) - \nu(m_{z_i,w_j,2\pi
\ell/L}(\widetilde{m}_i
(p_k)))\,\Big|~;\label{e:discrete_approx__d_mu_nu(z_i,w_j)_bis}
\end{eqnarray}
the error made in approximation
(\ref{e:discrete_approx__d_mu_nu(z_i,w_j)_bis}) is therefore
proportional to $L^{-1}+C\varphi\left( \set{p_k} \right)$.

To summarize, our approximation $T^R_{\wh{d}}(\mu_Z,\nu_W)$ to the
uniformly continuous $T^R_d(\mu,\nu)$ is based on two
approximations: on the one hand, we compute the transportation cost
between the discrete measures $\mu_Z,\nu_W$, approximating
$\mu,\nu$; on the other hand, this transportation cost involves a
local distance $\wh{d}^R_{\mu,\nu}$ which is itself an
approximation.
The transportation between the discrete measures will be computed by
solving a linear programming optimization, as explained in detail in
the next subsection.
The final approximation error (\ref{e:final_approx_error}) depends
on two factors: 1) the fill distances $\varphi(Z),\varphi(W)$ of the
sample sets $Z,W$, and 2) the approximation of the local distance
function $d^R_{\mu,\nu}(z_i,w_j)$ between the sample points.
Combining the discretization of the M\"{o}bius search with
(\ref{e:discrete_approx__d_mu_nu(z_i,w_j)_bis}), the total
approximation error is thus proportional to
$\omega_{d^R_{\Gamma_\mu,\Gamma_\nu}}\parr{\varphi\left( \set{p_k}
\right)} + L^{-1} + \varphi\left( \set{p_k} \right)$.

Recall that we are in fact using $\Gamma_\mu,\Gamma_\nu$ in the role
of of $\mu,\nu$ (see above), which entails an additional
approximation error. This error relates to the accuracy with which
discrete meshes approximate smooth manifolds, as well as the method
used to approximate uniformization. We come back to this question in
Appendix B. As far as we are aware, a full convergence result for
(any) discrete uniformization is still unknown; in any case, we
expect this error to be negligible (and approximately of the order
of the largest edge in the full mesh) compared to the others.

\subsection{Optimization via linear programming}

The discrete formulation of
eq.~(\ref{e:generalized_Kantorovich_transportation}) is commonly
formulated as follows:
\begin{equation}\label{e:discrete_kantorovich}
  \sum_{i,j}d_{ij}\pi_{ij} \rightarrow \min
\end{equation}
\begin{equation}\label{e:discrete_kantorovich_CONSTRAINTS}
\begin{array}{l}
  \left \{
  \begin{array}{l}
    \sum_i \pi_{ij} = \nu_j \\
    \sum_j \pi_{ij} = \mu_i \\
    \pi_{ij} \geq 0
  \end{array}
  \right . ,
\end{array}
\end{equation}
where $\mu_i=\mu(\Xi_i)$ and $\nu_j=\nu(\Upsilon_j)$, and $d_{ij} =
d^R_{\mu,\nu}(z_i,w_j)$.

In practice, surfaces are often only partially isometric (with a
large overlapping part), or the sampled points may not have a good
one-to-one and onto correspondence (i.e. there are points both in
$Z$ and in $W$ that do not correspond well to any point in the other
set). In these cases it is desirable to allow the algorithm to
consider transportation plans $\pi$ with marginals \emph{smaller or
equal} to $\mu$ and $\nu$. Intuitively this means that we allow that
only some fraction of the mass is transported and that the remainder
can be ``thrown away''. This leads to the following formulation:
\begin{equation}\label{e:discrete_PARTIAL_kantorovich}
\sum_{i,j}d_{ij}\pi_{ij} \rightarrow \min
\end{equation}
\begin{equation}\label{e:discrete_PARTIAL_kantorovich_CONSTRAINTS}
\begin{array}{l}
  \left \{ \begin{array}{c}
             \sum_i \pi_{ij} \leq \nu_j \\
             \sum_j \pi_{ij} \leq \mu_i \\
             \sum_{i,j} \pi_{ij} = Q  \\
             \pi_{ij} \geq 0
           \end{array}
   \right .
\end{array}
\end{equation}
where $0 < Q \leq 1$ is a parameter set by the user that indicates
how much mass \emph{must} be transported, in total.

The corresponding transportation distance is defined by
\begin{equation}\label{e:discrete_trans_dist}
    T_d(\nu,\nu) = \sum_{ij}d_{ij}\pi_{ij},
\end{equation}
where $\pi_{ij}$ are the entries in the matrix $\pi$ for the optimal
(discrete) transportation plan.

Since these equations and constraints are all linear, we have the
following theorem:
\begin{thm}
The equations {\rm
(\ref{e:discrete_kantorovich})-(\ref{e:discrete_kantorovich_CONSTRAINTS})}
and {\rm (\ref{e:discrete_PARTIAL_kantorovich})-
(\ref{e:discrete_PARTIAL_kantorovich_CONSTRAINTS})} admit a global
minimizer that can be computed in polynomial time, using standard
linear-programming techniques.
\end{thm}

When correspondences between surfaces are sought, i.e. when one
imagines one surface as being transformed into the other, one is
interested in restricting $\pi$ to the class of permutation matrices
instead of allowing all bistochastic matrices. (This means that each
entry $\pi_{ij}$ is either 0 or 1.)  In this case the number of
centers $z_i$ must equal that of $w_j$, i.e. $n=N=p$, and it is best
to pick the centers so that $\mu_i=\frac{1}{N}=\nu_j$, for all $i,\
j$. It turns out that this is sufficient to {\em guarantee} (without
restricting the choice of $\pi$ in any way) that the minimizing
$\pi$ is a permutation:

\begin{thm}
If $n=N=p$ and $\mu_i=\frac{1}{N}=\nu_j$, then
\begin{enumerate}
\item
There exists a global minimizer of
{\rm(\ref{e:discrete_kantorovich})} that is a permutation matrix.
\item
If furthermore $Q = \frac{M}{N}$, where $M< N$ is an integer, then
there exists a  global minimizer of {\rm
(\ref{e:discrete_PARTIAL_kantorovich})} $\pi$ such that $\pi_{ij}
\in \{0,1\}$ for each $i,\,j$.
\end{enumerate}
\label{t:relaxation}
\end{thm}
\begin{rem}
In the second case, where $\pi_{ij} \in \{0,1\}$ for each $i,\,j$
and $\sum_{i,j=1}^N \pi_{ij}=M$, $\pi$ can still be viewed as a
permutation of $M$ objects, ``filled up with zeros''.  That is, if
the zero rows and columns of $\pi$ (which must exist, by the pigeon
hole principle) are removed, then the remaining $M \times M$ matrix
is a permutation.
\end{rem}
\begin{proof}
We first note that in both cases, we can simply renormalize each
$\mu_i$ and $\nu_j$ by $N$, leading to the rescaled systems
\begin{equation}
  \left \{ \begin{array}{c}
             \sum_i \pi_{ij} = 1 \\
             \sum_j \pi_{ij} = 1 \\
             \pi_{ij} \geq 0
           \end{array}
   \right . \mbox{\hspace{1 in}}
  \left \{\begin{array}{c}
         \sum_i \pi_{ij} \leq 1 \\
             \sum_j \pi_{ij} \leq 1 \\
             \sum_{i,j} \pi_{ij} = M  \\
             \pi_{ij} \geq 0
\end{array}
   \right .
\label{e:discrete_kantorovich_CONSTRAINTS_disk}
\end{equation}
To prove the first part, we note that the left system in
(\ref{e:discrete_kantorovich_CONSTRAINTS_disk}) defines a convex
polytope in the vector space of matrices that is exactly the
Birkhoff polytope of bistochastic matrices. By the Birkhoff-Von
Neumann Theorem \cite{Lovasz86} every  bistochastic matrix is a
convex combination of the permutation matrices, i.e. each $\pi$
satisfying the left system in
(\ref{e:discrete_kantorovich_CONSTRAINTS_disk}) must be of the form
$\sum_k c_k\tau^k$, where the $\tau^k$ are the $N!$ permutation
matrices for $N$ objects, and $\sum_k c_k = 1$, with  $c_k \geq 0$.
The minimizing $\pi$ in this polytope for the linear functional
(\ref{e:discrete_kantorovich}) must thus be of this form as well. It
follows that at least one $\tau^k$ must also minimize
(\ref{e:discrete_kantorovich}), since otherwise we would obtain the
contradiction
\begin{equation}\label{e:linear_extrermas_at_vertices}
    \sum_{ij}d_{ij}\pi_{ij} = \sum_k c_k \Big (\sum_{ij}d_{ij} \tau^k_{ij} \Big ) \geq \min_k \Big \{ \sum_{ij}d_{ij} \tau^k_{ij} \Big \} > \sum_{i,j}\,d_{ij}\,\pi_{ij} ~.
\end{equation}

The second part can be proved along similar steps: the right system
in (\ref{e:discrete_kantorovich_CONSTRAINTS_disk}) defines a convex
polytope in the vector space of matrices; it follows that every
matrix that satisfies the system of constraints is a convex
combination of the extremal points of this polytope. It suffices to
prove that these extreme points are exactly those matrices that
satisfy the constraints and have entries that are either 0 or 1
(this is the analog of the Birkhoff-von Neumann theorem for this
case; we prove this generalization in a lemma in Appendix C); the
same argument as above then shows that there must be at least one
extremal point where the linear functional
(\ref{e:discrete_kantorovich}) attains its minimum.
\end{proof}

This means that, when we seek correspondences between two surfaces,
there is no need to {\em impose} the (very nonlinear) constraint on
$\pi$ that it be a permutation matrix; one can simply use a linear
program to solve either , with Theorem \ref{t:relaxation}
guaranteeing that the minimizer for the ``relaxed'' problem {\rm
(\ref{e:discrete_kantorovich})-(\ref{e:discrete_kantorovich_CONSTRAINTS})}
or {\rm (\ref{e:discrete_PARTIAL_kantorovich})-
(\ref{e:discrete_PARTIAL_kantorovich_CONSTRAINTS})} is of the
desired type if $n=N=p$ and $\mu_i=\frac{1}{N}=\nu_j$.

\subsection{Consistency}
In our schemes to compute the surface transportation distance, for
example by solving (\ref{e:discrete_PARTIAL_kantorovich}), we have
so far not included any constraints on the regularity of the
resulting optimal transportation plan $\pi^*$. When computing the
distance between a surface and a reasonable deformation of the same
surface, one does indeed find, in practice, that the minimizing
$\pi^*$ is fairly smooth, because neighboring points have similar
neighborhoods. There is no guarantee, however, that this has to
happen. Moreover, we will be interested in comparing surfaces that
are far from (almost) isometric, given by noisy datasets. Under such
circumstances, the minimizing $\pi^*$ may well ``jump around''. In
this subsection we propose a regularization procedure to avoid such
behavior.

Computing how two surfaces best correspond makes use of the values
of the ``distances in similarity'' $d^R_{\mu,\nu}(z_i,w_j)$ between
pairs of points that ``start'' on one surface and ``end'' on the
other; computing these values relies on finding a minimizing \Mbs
transformation for the functional (\ref{e:d_mu,nu(z,w)_def}). We can
keep track of these minimizing \Mbs transformations $m_{ij}$ for the
pairs of points $(z_i,w_j)$ proposed for optimal correspondence by
the optimal transport algorithm described above. Correspondence
pairs $(i,j)$ that truly participate in some close-to-isometry map
will typically have \Mbs transformations $m_{ij}$ that are very
similar. This suggests a method of filtering out possibly mismatched
pairs, by retaining only the set of correspondences $(i,j)$ that
cluster together within the M\"{o}bius group.

There exist many ways to find clusters. In our applications, we
gauge how far each \Mbs transformation $m_{ij}$ is from the others
by computing a type of $\ell_1$ variance:
\begin{equation}\label{e:variance_function}
    E_V(i,j) = \sum_{(k,\ell)}\norm{m_{ij} - m_{k\ell}},
\end{equation}
where the norm is the Frobenius norm (also called the
Hilbert-Schmidt norm) of the $2\times 2$ complex matrices
representing the M\"{o}bius transformations, after normalizing them
to have determinant one. We then use $E_V(i,j)$ as a consistency
measure of the corresponding pair $(i,j)$.

\section{Examples and comments}
\label{s:examples}
\begin{figure}[h]
\centering
\begin{tabular}{rcccc}
\includegraphics[width=0.2\columnwidth]{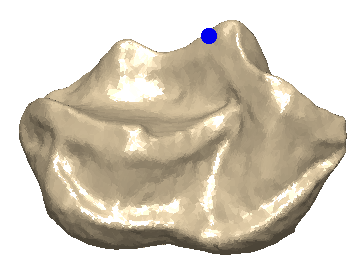} &
\includegraphics[width=0.2\columnwidth]{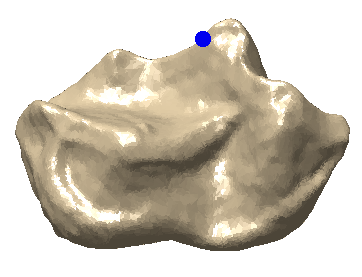} &
\includegraphics[width=0.2\columnwidth]{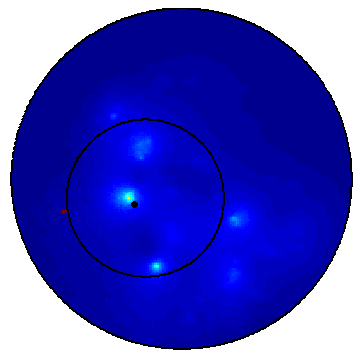} &
\includegraphics[width=0.2\columnwidth]{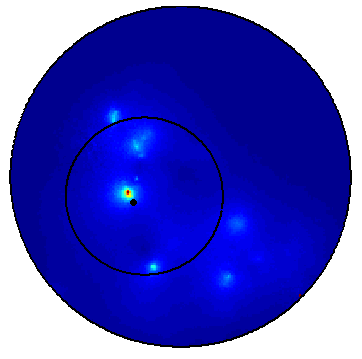} &
\includegraphics[width=0.2\columnwidth]{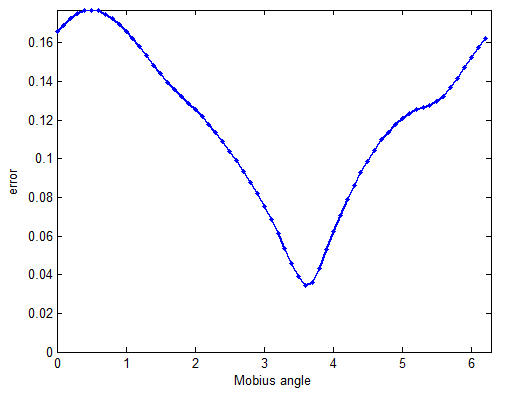} \\
(a) &&&& Good pair (a)\\
\includegraphics[width=0.2\columnwidth]{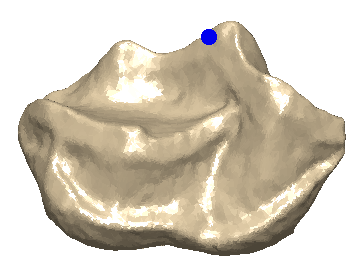} &
\includegraphics[width=0.2\columnwidth]{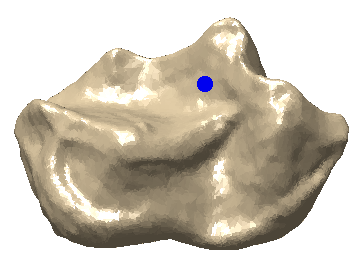} &
\includegraphics[width=0.2\columnwidth]{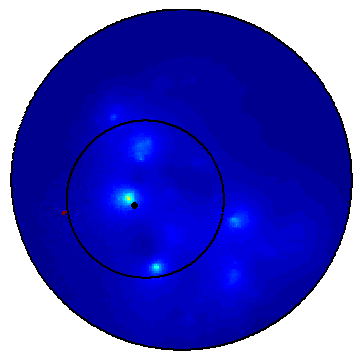} &
\includegraphics[width=0.2\columnwidth]{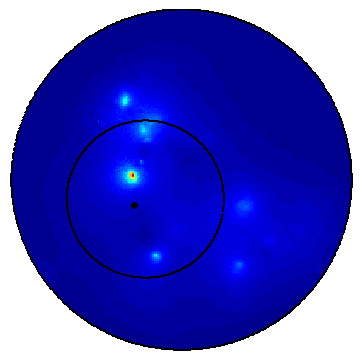} &
\includegraphics[width=0.2\columnwidth]{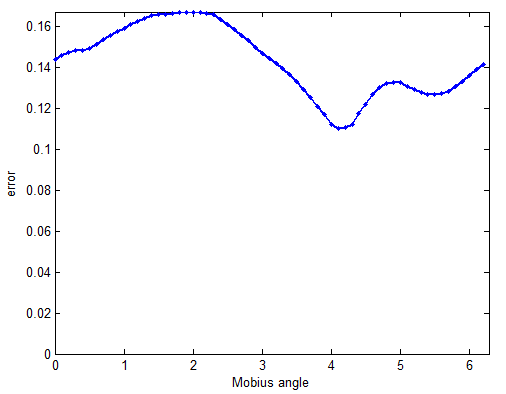} \\
(b) & & &  & Erroneous pair (b)
\end{tabular}
\caption{Calculation of the local distance
$d^R_{\mu,\nu}(\cdot,\cdot)$ between pairs of points on two
different surfaces (each row shows a different pair of points; the
two surfaces are the same in the top and bottom rows). The first row
shows a ``good'' pair of points together with the alignment of the
conformal densities $\mu,m^*\nu$ based on the best M\"{o}bius
transformation $m$ minimizing $\int_\D \norm{\mu-m^*\nu}d\vol_\M$.
The plot of this latter integral as a function of $m$ (parameterized
by $\sigma \in [0,2\pi)$, see (\ref{e:disk_mobius})) is shown in the
right-most column. The second row shows a ``bad'' correspondence
which indeed leads to a higher local distance
$d^R_{\mu,\nu}$.}\label{fig:good_bad_pair_correspondence}

\end{figure}

In this section we present a few experimental results using our new
surface comparison operator. These concern an application to
biology; in a case study of the use of our approach to the
characterization of mammals by the surfaces of their molars, we
compare high resolution scans of the masticating surfaces of molars
of several lemurs, which are small primates living in Madagascar.
Traditionally, biologists specializing in this area carefully
determine landmarks on the tooth surfaces, and measure
characteristic distances and angles involving these landmarks. A
first stage of comparing different tooth surfaces is to identify
correspondences between landmarks.  Figure
\ref{fig:good_bad_pair_correspondence} illustrates how
$d^R_{\mu,\nu}(z,w)$ can be used to find corresponding pairs of
points on two surfaces by showing both a ``good'' and  a ``bad''
corresponding pair. The left two columns of the figure show the pair
of points in each case; the two middle columns show the best fit
after applying the minimizing M\"{o}bius on the corresponding disk
representations; the rightmost column plots $ \int_{\Omega_{z_0,R}}
\,|\,\mu(z) - (m_{z_0,w_0,\sigma}^*\nu)(z)\,|\, d\vol_H(z)$, the
value of the ``error'', as a function of parameter $\sigma$,
parametrizing the \Mbs transformations that map a give point $z_0$
to another given point $w_0$ (see Lemma
\ref{lem:a_and_tet_formula_in_mobius_interpolation}). The ``best''
corresponding point $w_0$ for a given $z_0$ is the one that produces
the lowest minimal value for the error, i.e. the lowest
$d^R_{\mu,\nu}(z_0,w_0)$.

Figure \ref{fig:120_corrs} show the top 120 most consistent
corresponding pairs (in groups of 20) for two molars belonging to
lemurs of different species. Corresponding pairs are indicated by
highlighted points of the same color. These correspondences have
surprised the biologists from whom we obtained the data sets; their
experimental measuring work, which incorporates finely balanced
judgment calls, had defied earlier automatizing attempts.

Once the differences and similarities between molars from different
animals have been quantified, they can be used (as part of an
approach) to classify the different individuals. Figure
\ref{fig:distance_graph_embedded} illustrates a preliminary result
from \cite{Daubechies10} that illustrates the possibility of such
classifications based on the distance operator between surfaces
introduced in this paper. The figure illustrates the pairwise
distance matrix for eight molars, coming from individuals in four
different species (indicated by color). The clustering was based on
only the distances between the molar surfaces; it clearly agrees
with the clustering by species, as communicated to us by the
biologists from whom we obtained the data sets.

One final comment regarding the computational complexity of our
method. There are two main parts: the preparation of the distance
matrix $d_{ij}$ and the linear programming optimization. For the
linear programming part we used a Matlab interior point
implementation with $N^2$ unknowns, where $N$ is the number of
points spread on the surfaces. In our experiments, the optimization
typically terminated after $15-20$ iterations for $N=150-200$
points, which took about 2-3 seconds. The computation of the
similarity distance $d_{ij}$ took longer, and was the bottleneck in
our experiments. If we spread $N$ points on each surface, and use
them all (which was usually not necessary) to interpolate the
conformal factors $\Gamma_\mu, \Gamma_\nu$, if we use $P$ points in
the integration rule, and take $L$ points in the M\"{o}bius
discretization (see Section \ref{s:the_discrete_case_implementation}
for details) then each approximation of  $d^R_{\mu,\nu}(z_i,w_j)$ by
(\ref{e:discrete_approx__d_mu_nu(z_i,w_j)_bis}) requires $O(L \cdot
P \cdot N)$ calculations, as each evaluation of
$\Gamma_\mu,\Gamma_\nu$ takes $O(N)$ and we need $L\cdot P$ of
those. Since we have $O(N^2)$ distances to compute, the computation
complexity for calculating the similarity distance matrix $d_{ij}$
is $O(L\cdot P \cdot N^3)$. In practice this step was the most time
consuming and took around two hours for $N=300$. However, we have
not used any code optimization and we believe these times can be
reduced significantly.


\begin{figure}[h]
\centering
\begin{tabular}{cccc}
\includegraphics[width=0.3\columnwidth]{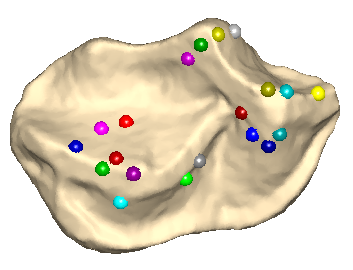} &
\includegraphics[width=0.3\columnwidth]{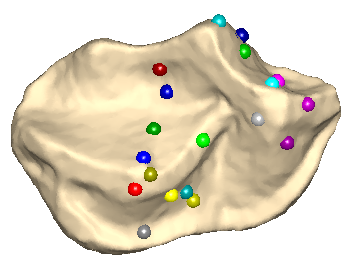} &
\includegraphics[width=0.3\columnwidth]{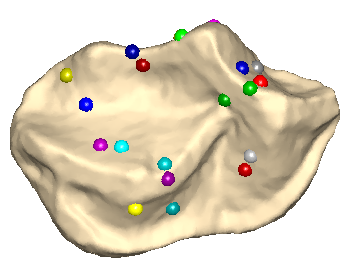} \\
\includegraphics[width=0.3\columnwidth]{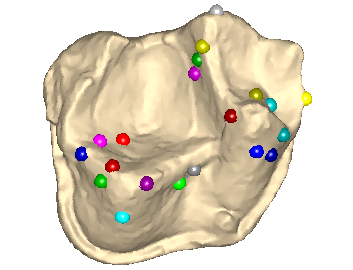} &
\includegraphics[width=0.3\columnwidth]{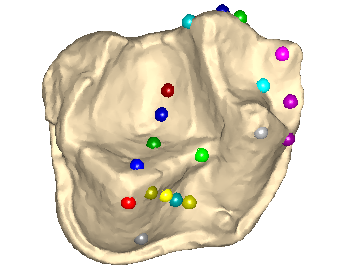} &
\includegraphics[width=0.3\columnwidth]{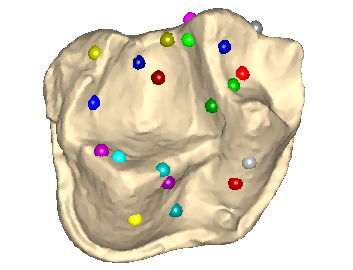} \\
\hline
\includegraphics[width=0.3\columnwidth]{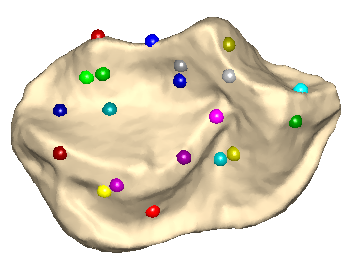} &
\includegraphics[width=0.3\columnwidth]{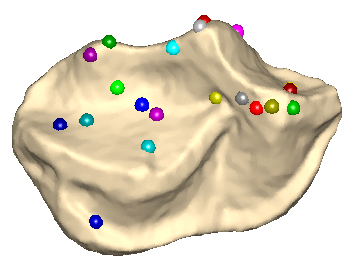} &
\includegraphics[width=0.3\columnwidth]{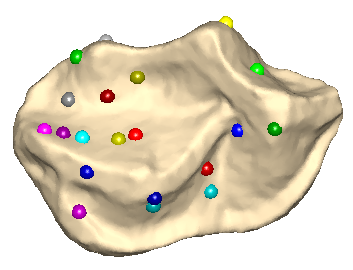} \\
\includegraphics[width=0.3\columnwidth]{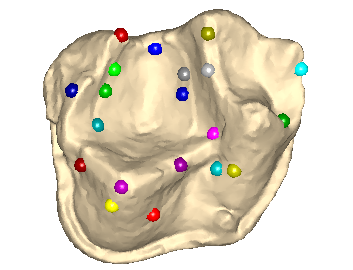} &
\includegraphics[width=0.3\columnwidth]{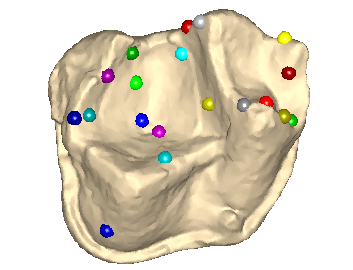} &
\includegraphics[width=0.3\columnwidth]{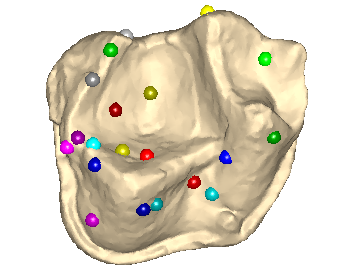} \\
\end{tabular}
\caption{The top 120 most consistent corresponding pairs between two
molar teeth models.} \label{fig:120_corrs}
\end{figure}

\begin{figure}[h]
\centering
\includegraphics[width=0.9\columnwidth]{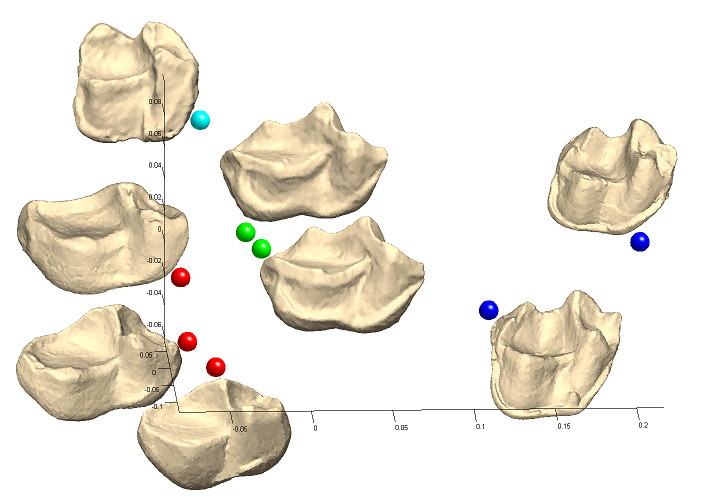}
\caption{Embedding of the distance graph of eight teeth models using
multi-dimensional scaling. Different colors represent different
lemur species. The graph suggests that the geometry of the teeth
might suffice to classify species.}
\label{fig:distance_graph_embedded}
\end{figure}


\section{Acknowledgments}
The authors would like to thank C\'{e}dric Villani and Thomas
Funkhouser for valuable discussions, and Jesus Puente for helping
with the implementation. We are grateful to Jukka Jernvall, Stephen
King, and Doug Boyer for providing us with the tooth data sets, and
for many interesting comments. ID gratefully acknowledges (partial)
support for this work by NSF grant DMS-0914892, and by an AFOSR
Complex Networks grant; YL thanks the Rothschild foundation for
postdoctoral fellowship support.

\bibliographystyle{amsplain}
\bibliography{kantorovich_for_surfaces}

\appendix
\renewcommand{\thesection}{ \Alph{section}}
\section{}
\label{a:appendix A}
\renewcommand{\thesection}{\Alph{section}}
This Appendix contains some technical proofs of Lemmas and Theorems
stated in section \ref{s:optimal_vol_trans_for_surfaces}, and
\ref{s:the_discrete_case_implementation}. We start with proving the
list of properties of the distance function $d^R_{\mu,\nu}(z,w)$
given in Theorem  \ref{thm:properties_of_d}:

{\bf Theorem} \ref{thm:properties_of_d}

{\em The distance function $d^R_{\mu,\nu}(z,w)$ satisfies the
following properties
\begin{table}[ht]
\begin{tabular}{c l l}
{\rm (1)} & $~d^R_{m^*_1\mu,m^*_2\nu}(m^{-1}_1(\z),m^{-1}_2(\w)) = d^R_{\mu,\nu}(\z,\w)~$ & {\rm Invariance under (well-defined)}\\
& & {\rm M\"{o}bius changes of coordinates} \\
&&\\
{\rm (2)} & $~d^R_{\mu,\nu}(\z,\w) = d^R_{\nu,\mu}(\w,\z)~$ & {\rm Symmetry} \\
&&\\
{\rm (3)} & $~d^R_{\mu,\nu}(\z,\w) \geq 0~$ & {\rm Non-negativity} \\
&&\\
{\rm (4)} & \multicolumn{2}{c}{$\!\!\!\!\!\!~d^R_{\mu,\nu}(\z,\w) = 0 \,\Longrightarrow \, \Omega_{z_0,R}$ {\rm in} $(\D,\mu)$ {\rm and} $\Omega_{w_0,R}$ {\rm in} $(\D,\nu)$ {\rm are isometric} }\\
&&\\
{\rm (5)} & $~d^R_{m^*\nu, \nu}(m^{-1}(\z),\z)=0~$ & {\rm Reflexivity} \\
&&\\
{\rm (6)} & $~d^R_{\mu_1,\mu_3}(z_1,z_3) \leq
d^R_{\mu_1,\mu_2}(z_1,z_2) + d^R_{\mu_2,\mu_3}(z_2,z_3)~$ & {\rm
Triangle inequality}
\end{tabular}
\end{table}
}
\begin{proof}

For (1), denote $m_1^{-1}(z_0)=z_1$, and $m_2^{-1}(w_0)=w_1$. Then
\begin{align*}
d^R_{m_1^* \mu, m_2^* \nu}(z_1,w_1) &=
\mathop{\inf}_{m(z_1)=w_1} \int_{\Omega_{z_1,R}}|m_1^*\mu(z)-m^*m_2^*\nu(z)|d\vol_H(z) \\
&= \mathop{\inf}_{m(z_1)=w_1}
\int_{\Omega_{z_1,R}}|\mu(m_1(z))-\nu(m_2(m(z)))|d\vol_H(z).
\end{align*}
Next set $\widetilde{m} = m_2 \circ m \circ m_1^{-1}$. Note that
$\widetilde{m}(z_0)=w_0$. Plugging $m_2(m(z))=\widetilde{m}(m_1(z))$
into the integral and carrying out the change of variables
$m_1(z)=z'\,$, we obtain
$$
\mathop{\inf}_{m(z_1)=w_1} \int_{\Omega_{z_1,R}}
|\,\mu(z')-\nu(\widetilde{m}(z'))\,|d\vol_H(z') =
\mathop{\inf}_{\widetilde{m}(z_0)=w_0}
\int_{\Omega_{z_0,R}}|\,\mu(z')-\nu(\widetilde{m}(z'))\,|\,d\vol_H(z').
$$


For (2), we use Lemma \ref{lem:symmetry_of_d_integral} and equations
(\ref{e:pullback_of_metric_density_mu_by_mobius}),
(\ref{e:push_forward_of_metric_density}) to write
\begin{align*}
d^R_{\mu,\nu}(z_0,w_0)&=
\mathop{\inf}_{m(z_0)=w_0} \int_{\Omega_\z,R} |\mu(z)-m^*\nu(z)|d\vol_H(z) \\
&= \mathop{\inf}_{m(z_0)=w_0} \int_{\Omega_\w,R}
|(m^{-1})^*\mu(w)-\nu(w)|d\vol_H(w) = d^R_{\nu,\mu}(w_0,z_0).
\end{align*}


(3) and (4) are immediate from the definition of $d^R_{\mu,\nu}$.


(5) follows from the observation that the minimizing $m$ (in the
definition (\ref{e:d_mu,nu(z,w)_def}) of $d^R_{\mu,\nu}$) is $m_1$
itself, for which the integrand, and thus the whole integral
vanishes identically.


For (6), let $m_1$ be a M\"{o}bius transformation such that
$m_1(z_1)=z_2$, and $m_2$ such that $m_2(z_2)=z_3$. Setting $m=m_2
\circ m_1$, we have
\begin{align}
d^R_{\mu_1,\mu_3}(z_1,z_3) &\leq \int_{\Omega_{z_1,R}}
|\,\mu_1(z) - m^* \mu_3(z)\,|\,d\vol_H(z) \nonumber\\
&\leq \int_{\Omega_{z_1,R}}|\,\mu_1(z) - m_1^*
\mu_2(z)\,|\,d\vol_H(z) + \int_{\Omega_{z_1,R}}|\,m_1^* \mu_2(z) -
m^* \mu_3(z)\,|\,d\vol_H(z)\,. \label{intermed}
\end{align}

The second term in (\ref{intermed}) can be rewritten as (using Lemma
\ref{lem:symmetry_of_d_integral}, the change of coordinates
$m_1(z_1)=z_2$ and the observation $m^*=m_1^* m_2^*$)
\begin{align*}
\int_{\Omega_{z_1,R}}|\,m_1^* \mu_2(z) - m^* \mu_3(z)\,|\,d\vol_H(z)
&=\int_{\Omega_{z_2,R}} |\,m_{1*}m_1^*\mu_2(w) - m_{1*}m_1^* m_2^*
\mu_3(w)\,|\,
d\vol_H(w)\\
&=\int_{\Omega_{z_2,R}} |\,\mu_2(w) - m_2^* \mu_3(w)\,|\,d\vol_H(w).
\end{align*}
We have thus
\[
d^R_{\mu_1,\mu_3}(z_1,z_3) \leq \int_{\Omega_{z_1,R}}|\,\mu_1(z) -
m_1^* \mu_2(z)\,|\,d\vol_H(z) + \int_{\Omega_{z_2,R}} |\,\mu_2(w) -
m_2^* \mu_3(w)\,|\,d\vol_H(w)~,
\]
and this for any $m_1,\,m_2 \in \Md$ such that $m_1(z_1)=z_2$ and
$m_2(z_2)=z_3$. Minimizing over $m_1$ and $m_2$ then leads to the
desired result.

\end{proof}

Next we prove the continuity properties of the function
$\Phi(z_0,w_0,\sigma) =
\int_{\Omega(z_0,R)}\,|\,\mu(z)-\nu(m_{z_0,w_0,\sigma}(z))\,|\,d\vol_H(z)$,
stated in Lemma \ref{lem:auxiliary}, which were used to prove
continuity of $d^R_{\mu,\nu}$ itself (in Theorem
\ref{thm:continuity_of_d^R_mu,nu}).

{\bf Lemma \ref{lem:auxiliary}} \\
{\em $\bullet$ For each fixed $(z_0,w_0)$ the function $\Phi(z_0,w_0,\cdot)$ is continuous on $S_1$.\\
$\bullet$ For each fixed $\sigma \in S_1$, $
\Phi(\cdot,\cdot,\sigma) $ is continuous on $\D \times \D$.
Moreover, the family $ \Big(\Phi(\cdot,\cdot,\sigma)\Big)_{\sigma
\in S_1} $ is equicontinuous.}

\begin{proof}
We start with the continuity in $\sigma$. We have
\[
\left|\,\Phi(z_0,w_0,\sigma)-\Phi(z_0,w_0,\sigma')\,\right| \leq
\int_{\Omega(z_0,R)}\,|\nu(m_{z_0,w_0,\sigma}(z))-\nu(m_{z_0,w_0,\sigma'}(z))\,|
\,d\vol_H(z)~.
\]
Because $\nu$ is continuous on $\D$, its restriction to the compact
set $\overline{\Omega(w_0,R)}$ (the closure of $\Omega(w_0,R)$) is
bounded. Since the hyperbolic volume of $\Omega(z_0,R)$ is finite,
the integrand is dominated, uniformly in $\sigma'$, by an integrable
function.  Since $m_{z_0,w_0,\sigma}(z)$ is obviously continuous in
$\sigma$, we can use the dominated convergence theorem to conclude.

Since $S^1$ is compact, this continuity implies that the infimum in
the definition of $d^R_{\mu,\nu}$ can be replaced by a minimum:
\[
d^R_{\mu,\nu}(z_0,w_0)=\mathop{\min}_{m(z_0)=w_0}\,
\int_{\Omega(z_0,R)}\,|\,\mu(z)-\nu(m(z))\,|\,d\vol_H(z)~.
\]

Next we prove continuity in $z_0$ and $w_0$ (with estimates that are
uniform in $\sigma$).

Consider two pairs of points, $(z_0,w_0)$ and $(z'_0,w'_0) \in \D
\times \D$. Then
\begin{align}
&|\,\Phi(z_0,w_0,\sigma)-\Phi(z'_0,w'_0,\sigma)\,|\nonumber\\
&~~~~~=\,\left|\,
\int_{\Omega(z_0,R)}\,|\,\mu(z)-\nu(m_{z_0,w_0,\sigma}(z))\,|\,d\vol_H(z)
- \int_{\Omega(z'_0,R)}\,|\,\mu(u)-\nu(m_{z'_0,w'_0 \sigma}(u)\,|\,d\vol_H(u)\,\right|\nonumber\\
&~~~~~\leq\,
\left|\,\int_{\Omega(z_0,R)}\,|\,\mu(z)-\nu(m_{z_0,w_0,\sigma}(z))\,|\,d\vol_H(z)
-
\int_{\Omega(z_0,R)}\,|\,\mu(m_{z_0,z'_0,1}z)-\nu(m_{z'_0,w'_0,\sigma}
\circ
 m_{z_0,z'_0,1}(z))\,|\,d\vol_H(u)\,\right|\nonumber\\
&~~~~~ \leq \int_{\Omega(z_0,R)}\,
\left(\,|\,\mu(z)-\mu(m_{z_0,z'_0,1}(z))\,| +
|\,\nu(m_{z_0,w_0,\sigma}(z))-\nu(m_{z'_0,w'_0,\sigma}(m_{z_0,z'_0,1}(z)))\,|\,\right)\,
d\vol_H(z) ~.\nonumber \label{interm2}
\end{align}
On the other hand, note that for any $\gamma >0$, $\mu$ and $\nu$
are continuous on the closures of $\Omega(z_0,R+\gamma)$ and
$\Omega(w_0,R+\gamma)$, respectively; since these closed hyperbolic
disks are compact, $\mu$ and $\nu$ are  bounded on these sets. Pick
now $\rho>0$ such that $|z'_0-z_0|<\rho$, $|w'_0-w_0|<\rho$ imply
that $\Omega(z'_0,R)\subset \Omega(z_0,R+\gamma)$ as well as
$\Omega(w'_0,R)\subset \Omega(w_0,R+\gamma)$. It follows that, if
$|z'_0-z_0|<\rho$ and $|w'_0-w_0|<\rho$, then
$|\,\mu(z)-\mu(m_{z_0,z'_0,1}(z)\,|$ and
$|\,\nu(m_{\z_0,w_0,\sigma}(z))-
\nu(m_{z'_0,w'_0,\sigma}(m_{z_0,z'_0,1}(z))) \,|$ are bounded
uniformly for $z \in \Omega(z_0,R)$. Since it is clear from the
explicit expressions (\ref{e:a_of_mobius}) that
$m_{z_0,z'_0,1}(z)\rightarrow z$ and
$m_{z'_0,w'_0,\sigma}(m_{z_0,z'_0,1}(z)) \rightarrow
m_{z_0,w_0,\sigma}(z)$ as $z'_0 \rightarrow z_0$ and $w'_0
\rightarrow w_0$, we can thus invoke the dominated convergence
theorem again to prove continuity of $\Phi(\cdot,\cdot,\sigma)$.

To prove the equicontinuity, we first note that $\nu$ is uniformly
continuous on $\Omega(w_0,R)\cup\Omega(w'_0,R) $, since $\nu$ is
continuous on the compact set $\overline{\Omega(w_0,R+\gamma)}$,
which contains $\Omega(w_0,R)\cup\Omega(w'_0,R)$ for all $w'_0$ that
satisfy $|w'_0-w_0| \leq \rho$. This means that, given any $\eps
>0$, we can find $\delta >0$ such that $|\nu(w)-\nu(w')|\le \eps$
holds for all $w,\,w'$ that satisfy $w,\,w' \in
\Omega(w_0,R)\cup\Omega(w'_0,R)$ and $|w-w'|\leq \delta$. This
implies the desired equicontinuity if we can show that
$|m_{z_0,w_0,\sigma}(z)-m_{z'_0,w'_0,\sigma}(m_{z_0,z'_0,1}(z))|$
can be made smaller than $\delta$, uniformly in $\sigma \in S_1$, by
making
$|z'_0-z_0|+|w'_0-w_0|$ sufficiently small.\\
We first estimate $|m_{z_0,w_0,\sigma}(z)-m_{z_0,w'_0,\sigma}(z)|$.
With the notations of (\ref{e:a_of_mobius}), we have
\begin{align*}
a(z_0,w_0,\sigma)-a(z_0,w'_0,\sigma)&=
\frac{(z_0-w_0\overline{\sigma})(1-\overline{z_0}w'_0\overline{\sigma})-
(z_0-w'_0\overline{\sigma})(1-\overline{z_0}w_0\overline{\sigma})}
{(1-\overline{z_0}w_0\overline{\sigma})(1-\overline{z_0}w'_0\overline{\sigma})}\\
&=\frac{(w_0-w'_0)\overline{\sigma}(|z_0|^2-1)}
{(1-\overline{z_0}w_0\overline{\sigma})(1-\overline{z_0}w'_0\overline{\sigma})}~,
\end{align*}
so that
\[
| a(z_0,w_0,\sigma)-a(z_0,w'_0,\sigma)|\leq \frac{|w_0-w'_0|}
{(1-|z_0|\,|w_0|)[1-|z_0|(|w_0|+\xi)]}\leq \frac{\xi}
{(1-|z_0|\,|w_0|)[1-|z_0|(|w_0|+\xi)]}
\]
when $|w_0-w'_0|<\xi$. It thus suffices to choose $\xi$ so that
$\xi< \zeta(1-|z_0|\,|w_0|)[1-|z_0|(|w_0|+\xi)]$ to ensure that $|
a(z_0,w_0,\sigma)-a(z_0,w'_0,\sigma)|<\zeta$. For the phase factor
$\tau$ in (\ref{e:a_of_mobius}) we obtain
\begin{align*}
 \tau(z_0,w_0,\sigma)-\tau(z_0,w'_0,\sigma)&=
\sigma \,\frac
{(1-\overline{z_0}w'_0\overline{\sigma})(1-z_0\overline{w_0}\sigma)
-
(1-\overline{z_0}w_0\overline{\sigma})(1-z_0\overline{w'_0}\sigma)}
{(1-\overline{z_0}w_0\overline{\sigma})(1-\overline{z_0}w'_0\overline{\sigma})}\\
&=\sigma \, \frac
{(w_0-w'_0)\overline{z_0}\overline{\sigma}-(\overline{w_0}-\overline{w'_0})z_0\sigma
+|z_0|^2(\overline{w_0}w'_0-\overline{w'_0}w_0)}
{(1-\overline{z_0}w_0\overline{\sigma})(1-\overline{z_0}w'_0\overline{\sigma})}\\
&=\sigma \, \frac {(w_0-w'_0)\overline{z_0}\overline{\sigma}-
z_0(\overline{w_0}-\overline{w'_0})\sigma +
|z_0|^2[\overline{w_0}(w'_0-w_0)+w_0(\overline{w_0}-\overline{w'_0}])
}
{(1-\overline{z_0}w_0\overline{\sigma})(1-\overline{z_0}w'_0\overline{\sigma})}~;
\end{align*}
when $\abs{w_0-w'_0}<\xi$, this implies
\[
|\tau(z_0,w_0,\sigma)-\tau(z_0,w'_0,\sigma)| \leq
\frac{|z_0|\,|w_0|\,[2+|z_0|(2|w_0|+\xi)]}
{(1-|z_0|\,|w_0|)[1-|z_0|(|w_0|+\xi)]}\xi~,
\]
which can clearly be made smaller than any $\zeta >0$ by choosing
$\xi$ sufficiently small. All this implies that (use
(\ref{e:a_of_mobius}))
\begin{align*}
|m_{z_0,w_0,\sigma}(z)-m_{z_0,w'_0,\sigma}(z)|
&\leq|\tau(z_0,w_0,\sigma)-\tau(z_0,w'_0,\sigma)|\frac{1+|z|}{1-|z|}
\,+\,|a(z_0,w_0,\sigma)-a(z_0,w'_0,\sigma)|\frac{(1+|z|)^2}
{(1-|z|)^2}\\
&\leq \zeta \,\frac{2(1+|z|)}{(1-|z|)^2},
\end{align*}
which will be  smaller than $\delta/2$, uniformly in $\sigma$, if
$\zeta < \delta (1-|z|^2)/8$; this bound on $\zeta$ in turn
determines the bound to be imposed on the $\xi$ used above. Hence
$|m_{z_0,w_0,\sigma}(z)-m_{z_0,w'_0,\sigma}(z)|<\delta/2$ can be
guaranteed, uniformly in $\sigma$, by choosing $|w_0-w'_0|<\xi$
for sufficiently small $\xi$. \\
One can estimate likewise
\[
|m_{z_0,w'_0,\sigma}(z)-m_{z'_0,w'_0,\sigma}(m_{z'_0,z_0,1}(z))|~,
\]
and show that this too can be made smaller than $\delta/2$,
uniformly in $\sigma$, by imposing sufficiently tight bounds on
$|z'_0-z_0|$ and $|w'_0-w_0|$. Combining all these estimates then
leads to the desired equicontinuity, as indicated earlier.
\end{proof}

To prove Lemma \ref{lem:extension_of_d}, we shall use the following
lemma:\\
\begin{lem}\label{lem:extermal_mobius}
Consider $u_k=e^{\bfi\psi} + \eps_k$, where $\abs{\eps_k}\too 0$ as
$k\too \infty$. Then there exists, for every $\eps>0$, a $K \in
\mathds{N}$ such that for all $k>K$; and all $\wh{m} \in
M_{D,0,u_k}$,
$$\inf_{w\in \Omega_{0,R}}\abs{\wh{m}(w)} > 1-\eps.$$
\end{lem}
The set $M_{D,0,u_k}$ used in this lemma is given by Definition
\ref{def:M_D,z_0,w_0}.
\begin{proof}
From Lemma \ref{lem:a_and_tet_formula_in_mobius_interpolation} we
can write $\wh{m}$ as $$\wh{m}(w)=e^{\bfi \theta}\frac{w+u_k
e^{-\bfi \theta}}{1 + \bbar{u_k}e^{\bfi \theta}w},$$ for some
$\theta \in [0,2\pi)$. Substituting $u_k = e^{\bfi \psi} + \eps_k$
in this equation we get
\begin{align*}
\wh{m}(w) & = e^{\bfi \theta}\frac{w+(e^{\bfi \psi} + \eps_k)
e^{-\bfi \theta}}{1 + \bbar{(e^{\bfi \psi} + \eps_k)}e^{\bfi
\theta}w}
 = e^{\bfi \psi} \frac{1+w e^{\bfi (\theta - \psi)} + \eps_k \eps^{-\bfi \psi}}{1 + w e^{\bfi (\theta - \psi)} + \bbar{\eps_k}e^{\bfi\theta}w }.
\end{align*}
Writing the shorthand $s$ for $s = 1+w e^{\bfi(\theta-\psi)}$, we
have thus
\begin{align*}
\abs{\wh{m}(w)-e^{\bfi \psi}} & = \abs{e^{\bfi \psi} \frac{s +
\eps_k \eps^{-\bfi \psi}}{s + \bbar{\eps_k}e^{\bfi\theta}w } -
e^{\bfi \psi}}
 \leq
\abs{ e^{\bfi \psi}\frac{\eps_k e^{-\bfi \psi} - \bbar{\eps}_k e^{\bfi \theta}w }{s + \bbar{\eps}_k e^{\bfi \theta }w }  } \\
& \leq \frac{\abs{ \eps_k e^{-\bfi \psi} - \bbar{\eps}_k e^{\bfi
\theta}w  }}{\abs{ s + \bbar{\eps}_k e^{\bfi \theta}w }}
 \leq
\frac{\abs{\eps_k}\parr{1+ |w|}}{|s|-|\eps_k||w|}
\end{align*}
Now for all $w\in \Omega_{0,R}$, $|w|<r_R=\tanh^{-1}(R)$. This
implies $|s| \geq 1 - |w| \geq 1-r_R$, and $1+|w| \leq 1+r_R$, so
that
$$\abs{\wh{m}(w)-e^{\bfi \psi}} \leq |\eps_k|\frac{1+r_R}{1-r_R - |\eps_k|r_R} = |\eps_k|\frac{1+r_R}{1-r_R\parr{1+ |\eps_k|}}.$$
Since $\abs{\eps_k}\too 0$ the lemma follows.
\end{proof}


We are now ready for \\
{\bf Lemma \ref{lem:extension_of_d}} {\em Let
$\set{(z_k,w_k)}_{k\geq 1} \subset \D\times\D$ be a sequence that
converges, in the Euclidean norm, to some point $(z',w') \in
\bbar{\D}\times\bbar{\D} \setminus \D\times\D$, that is
$|z_k-z'|+|w_k-w'| \too 0$, as $k \too \infty$. Then, $\lim_{k\too
\infty}d^R_{\xi,\zeta}(z_k,w_k)$ exists and depends only on the
limit point $(z',w')$.}
\begin{proof}
Since $(z',w') \in \bbar{\D}\times\bbar{\D} \setminus \D\times\D$
either $z'\in \bbar{\D}\setminus \D$ or $w'\in \bbar{\D}\setminus
\D$. Let us assume that $z'\in \bbar{\D}\setminus \D$ (the case
$w'\in \bbar{\D}\setminus \D$ is similar). Denote by $m_k$ an
arbitrary M\"{o}bius transformation in $M_{D,0,w_k}$. By symmetry of
the distance and using a change of variables we then obtain
\begin{align*}
d^R_{\xi,\zeta}(z_k,w_k) & = d^R_{\zeta,\xi}(w_k,z_k) \\
& = \min_{m(w_k)=z_k} \int_{\Omega_{w_k,R}} \Big |
\zeta(w)-\xi(m(w)) \Big |d\vol_H(w) \\ & = \min_{m(w_k)=z_k}
\int_{\Omega_{0,R}}\Big | \zeta(m_k(w))-\xi(m(m_k(w))) \Big
|d\vol_H(w).
\end{align*}
Now, recall that $\xi(z) = \xi^H(z) = \wt{\xi}(z) (1-|z|^2)^2$,
where $\wt{\xi}(z)$ is a bounded function, $\sup_{z\in
\D}|\wt{\xi}(z)| \leq C_{\wt{\xi}}$. From Lemma
\ref{lem:extermal_mobius} we know that for every $\eps>0$ and for
$k>K$ sufficiently large, $\abs{m(m_k((w))} > 1 -\eps$ for all $w\in
\Omega_{0,R}$, and all $m$ such that $m(w_k)=z_k$. This means that
for these $k>K$ we have
\begin{align*}
\abs{\xi(m(m_k(w)))} & =
\abs{\wt{\xi}(m(m_k(w)))}(1-\abs{m(m_k(w))}^2)^2 \\ \ & \leq
C_{\wt{\xi}}(1-(1-\eps)^2)^2\leq C_{\wt{\xi}}\eps^2(2-\eps)^2,
\end{align*}
for all $w \in \Omega_{0,R}.$
Therefore,
\begin{align*}
& \abs{d^R_{\xi,\zeta}(z_k,w_k) -
\int_{\Omega_{0,R}}\babs{\zeta(m_k(w))}d\vol_H(w)} \\ & \leq \abs{
\min_{m(w_k)=z_k}
\int_{\Omega_{0,R}}\babs{\zeta(m_k(w))-\xi(m(m_k(w)))}d\vol_H(w) -
\int_{\Omega_{0,R}}\babs{\zeta(m_k(w))}d\vol_H(w) } \\ &
 \leq
\abs{ \min_{m(w_k)=z_k} \int_{\Omega_{0,R}} \set{
\babs{\zeta(m_k(w))-\xi(m(m_k(w)))} - \babs{\zeta(m_k(w))} }
d\vol_H(w)} \\ & \leq \min_{m(w_k)=z_k} \int_{\Omega_{0,R}}
\babs{\xi(m(m_k(w)))} d\vol_H(w) \too 0, \ \mathrm{as} \ k\too
\infty.
\end{align*}
Therefore $d^R_{\xi,\zeta}(z_k,w_k)$ converges, as $k\too \infty$,
if and only if $\int_{\Omega_{0,R}}\abs{\zeta(m_k(w))}d\vol_H(w)$
converges, and to the same limit, for any $m_k \in M_{D,0,w_k}$. We
can take, for instance, $m_k(w) = \frac{w+w_k}{1+\bbar{w_k}w}$ which
gives
$$\int_{\Omega_{0,R}}\abs{\zeta(m_k(w))}d\vol_H(w)  =
\int_{\Omega_{0,R}}\abs{\zeta\parr{\frac{w+w_k}{1+\bbar{w_k}w}}}d\vol_H(w).$$
For $w\in \Omega_{0,R}$, $\abs{1+\bbar{w_k}w} > 1 - r_R$. It follows
that this expression has a limit as $k \too \infty$, and
$$\lim_{k\too \infty} \int_{\Omega_{0,R}}\abs{\zeta(m_k(w))}d\vol_H(w) =
\int_{\Omega_{0,R}}\abs{\zeta\parr{
\frac{w+w'}{1+\bbar{w'}w}}}d\vol_H(w),$$ which clearly depends on
$w'$, not on the sequence $\set{w_k}$.
\end{proof}

Next, we prove Theorem
\ref{thm:convergence_of_numerical_quadrature}. We start with a
simple lemma showing that all M\"{o}bius transformations restricted
to $\Omega_{0,R}$, $R < \infty$, are Lipschitz with a universal
constant, for which we provide an upper bound.
\begin{lem}\label{lem:mobius_is_lipschitz}
A M\"{o}bius transformation $m \in \Md$ restricted to
$\Omega_{0,R}$, $R<\infty$ is Lipschitz continuous with Lipschitz
constant $C_m \leq \frac{1-|a|^2}{(1-r_R|a|)^2}$.
\end{lem}
\begin{proof}
Denote $m(z)=e^{\bfi \theta}\frac{z-a}{1-z\bbar{a}}$. Then, for
$z,w\in \Omega_{0,R}$ we have
\begin{align*}
\babs{m(z)-m(w)} & \leq \babs{e^{\bfi \theta}\frac{z-a}{1-z\bbar{a}}
- e^{\bfi \theta}\frac{w-a}{1-w\bbar{a}} } \leq
\babs{\frac{(z-a)(1-w\bbar{a})-(w-a)(1-z\bbar{a})}{(1-z\bbar{a})(1-w\bbar{a})}}\\
& \leq \babs{\frac{(z-w)(1-\abs{a}^2)}{(1-z\bbar{a})(1-w\bbar{a})}}
\leq |z-w|\frac{1-\abs{a}^2}{(1-r_R\abs{a})^2}.
\end{align*}
\end{proof}

Next we prove: \\
{\bf Theorem \ref{thm:convergence_of_numerical_quadrature}} {\em For
continuously differentiable $\mu,\nu$,
\[
\left|\,d^R_{\mu,\nu}(z_i,w_j)- \min_{m(z_i)=w_j}\sum_k \,\alpha_k
\,\left|\,\mu(\widetilde{m}_i (p_k)) - \nu(m(\widetilde{m}_i
(p_k)))\,\right|\, \right| \leq C\,\varphi\left( \set{p_k} \right)~,
\]
where the constant $C$ depends only on $\mu,\nu,R$. }
\begin{proof}
First, denote $f(z) = \babs{\mu(\wt{m}_i (z)) - \nu(m(\wt{m}_i
(z)))}$. Then,
\begin{eqnarray}\label{e:approx_quad_eq1}
\babs{\int_{\Omega_0}f(z)d\vol_H(z)-\min_{m(z_i)=w_j}\sum_k \alpha_k
f(p_k)} & \leq \sum_k \int_{\Omega_0} \babs{f(z)-f(p_k)}d\vol_H(z)  \\
\nonumber & \leq \omega^{\Omega_0}_{f}\parr{\varphi
\parr{\set{p_k}}} \int_{\Omega_0}d\vol_H,
\end{eqnarray}
where the modulus of continuity $\omega^{\Omega_0}_{f}\parr{h} =
\sup_{|z-w|<h; z,w\in \Omega_0}\abs{f(z)-f(w)}$ is used. Note that
\begin{equation}\label{e:approx_quad_eq2}
    \omega^{\Omega_0}_{f} \leq \omega^{\Omega_0}_{\mu \circ \wt{m}_i} +
\omega^{\Omega_0}_{\nu \circ m \circ \wt{m}_i }.
\end{equation}
Since $\mu,\nu$ have continuous derivatives on compact domain, they
are Lipschitz continuous. Denote their Lipschitz constants by
$C_\mu,C_\nu$, respectively. From Lemma
\ref{lem:mobius_is_lipschitz} we see that, for $z,w \in \Omega_0$,
\begin{align*}
\babs{\mu(\wt{m}_i(z))-\mu(\wt{m}_i(w))} & \leq C_\mu
\babs{\wt{m}_i(z) - \wt{m}_i(w)}  \leq C_\mu
\frac{1-|a|^2}{(1-r_R|a|)^2} \abs{z-w}  \leq C_\mu
\frac{1}{(1-r_R)^2} \abs{z-w},
\end{align*}
which is independent of $\wt{m}_i$. Similarly,
\begin{align*}
\babs{\nu(m(\wt{m}_i(z)))-\nu(m(\wt{m}_i(w)))} & \leq C_\nu
\babs{m(\wt{m}_i(z)) - m(\wt{m}_i(w))} \leq C_\nu
\frac{1}{(1-r_R)^2} \abs{z-w},
\end{align*}
which is independent of $m,\wt{m}_i$. Combining these with eq.
(\ref{e:approx_quad_eq1}-\ref{e:approx_quad_eq2}) we get
$$\babs{\int_{\Omega_0}f(z)d\vol_H(z)-\min_{m(z_i)=w_j}\sum_k \alpha_k
f(p_k)} \leq
\parr{C_\mu+C_\nu}\frac{\int_{\Omega_0}d\vol_H}{(1-r_R)^2}\varphi \parr{\set{p_k}},$$
which finishes the proof.
\end{proof}

\renewcommand{\thesection}{ \Alph{section}}
\section{}
\label{a:appendix B}
\renewcommand{\thesection}{\Alph{section}}
In this appendix we provide a short exposition on discrete and
conjugate discrete harmonic functions on triangular meshes as
presented in \cite{Dziuk88,Pinkall93,Polthier00,Polthier05}, and we
show how this theory can be used in our context to conformally
flatten disk-type ( or even just simply connected) triangular
meshes.

We will use the same notations as in Section
\ref{s:the_discrete_case_implementation}. Discrete harmonic
functions are defined using a variational principle in the space of
continuous piecewise linear functions defined over the mesh $\PL_\M$
(\cite{Dziuk88}), as follows. Let us denote by $\phi_i(z),$
$i=1,..,m,$ the scalar functions that satisfy
$\phi_j(v_i)=\delta_{i,j}$ and are linear on each triangle
$f_{i,j,k} \in \F$. Then, the (linear) space of continuous
piecewise-linear function on $M$ can be written in this basis:
$$\PL_M = \set{\sum_{i=1}^m u_i\phi_i(z) \ \mid \ (u_1,...,u_m)^T\in \Real^m }.$$
Next, the following quadratic form is defined over $\PL_M$:
\begin{equation}\label{e:discrete_dirichlet_energy}
    E_{Dir}(u) = \sum_{f\in \F} \int_{f} \ip{\nabla u,\nabla u} d\vol_{\Real^3},
\end{equation}
where $\ip{\cdot}=\ip{\cdot}_{\Real^3}$ denotes the inner-product
induced by the ambient Euclidean space, and $d\vol_{\Real^3}$ is the
induced volume element on $f$. This quadratic functional, the {\em
Dirichlet energy}, can be written in coordinates of the basis
defined earlier as follows:
\begin{equation}\label{e:dirichelet_energy_in_coordinates}
    E_{Dir}\left(\sum_i u_i \phi_i \right) = \sum_{i,j=1}^m u_i u_j \brac{\sum_{f\in\F} \int_{f} \ip{\nabla \phi_i, \nabla \phi_j}}d\vol_{\Real^3} =
    \sum_{i,j=1}^m u_i u_j \int_{M} \ip{\nabla \phi_i, \nabla \phi_j}d\vol_{\Real^3}.
\end{equation}

The discrete harmonic functions are then defined as the functions
$u\in PL_M$ that are critical for $E_{Dir}(u)$, subject to some
constraints on the boundary of $M$. The linear equations for
discrete harmonic function $u \in \PL_M$ are derived by partial
derivatives of $E_{Dir}$, (\ref{e:dirichelet_energy_in_coordinates})
w.r.t. $u_i, i=1,..,m$:
\begin{equation}\label{e:partials_of_dirichlet}
    \frac{\partial E_{Dir}(u)}{\partial u_k} =
2\sum_{i=1}^m u_i \brac{\sum_{f\in\F} \int_{f} \ip{\nabla \phi_i,
\nabla \phi_k}}d\vol_{\Real^3} = 2 \int_M \ip{\nabla u , \nabla
\phi_k}d\vol_{\Real^3} = 2 \int_{R_k} \ip{\nabla u , \nabla
\phi_k}d\vol_{\Real^3},
\end{equation}
where $R_k \subset M$ is the 1-ring neighborhood of vertex $v_k$.
The last equality uses that $\phi_k$ is supported on $R_k$.

Now, let $u=\sum_i u_i \phi_i$ be a discrete harmonic function.
Pinkall and Polthier observed that conjugating the
piecewise-constant gradient field $\nabla u$ (constant on each
triangle $f \in \F$), i.e. rotating the gradient $\nabla u$ in each
triangle $f$ by $\pi/2$ in the positive ( = counterclockwise) sense
(we assume $M$ is orientable), results in a new vector field $*du =
J du$ with the special property that its integrals along (closed)
paths that cross edges only at their mid-points are systematically
zero (see for example \cite{Polthier05}). This means in particular
that we can define a piecewise linear function $\stu$ such that its
gradient satisfies $d\stu = *du$ and that is furthermore continuous
through the mid-edges $\mv \in \mV$. The space of piecewise-linear
functions on meshes that are continuous through the mid-edges is
well-known in the finite-element literature, where it is called
$\ncPL_M$, the space of non-conforming finite elements
\cite{Brenner:2008:MTF}. The Dirichlet form
(\ref{e:discrete_dirichlet_energy}) is defined over the space of
non-conforming elements $\ncPL_M$ as well; the non-conforming
discrete harmonic functions are defined to be the functions
$v\in\ncPL_M$ that are critical for $E_{Dir}$ and that satisfy some
constraints on the mid-edges of the boundary of the mesh. Polthier
\cite{Polthier05} shows that if $u\in \PL_M$ is a discrete harmonic
function, then $\stu \in \ncPL_M$ is also discrete harmonic, with
the same Dirichlet energy, and vise-versa. Solving for the discrete
harmonic function after fixing values at the boundaries amounts to
solving a sparse linear system which is explicitly given in
\cite{Polthier05}.

This theory can be used to define discrete conformal mappings, and
used to flatten a mesh in a ``discrete conformal'' manner, as
follows. The flattening is done by constructing a pair of conjugate
piecewise linear functions $(u,\stu)$ where $u\in \PL_M$, $\stu \in
\ncPL_M$, and the flattening map $\Phi:\mM \rightarrow \C$ is given
by
\begin{equation}\label{e:Phi_flattening_map}
    \Phi = u+\bfi \stu.
\end{equation}
Since $d\stu = J du$, $\Phi$ is a similarity transformation on each
triangle $f\in \F$. Furthermore, $\Phi$ is continuous through the
mid-edges $\mv_r\in\mV$. This means that $\Phi$ is well-defined on
the mid-edges $\mV$ and maps them to the complex plane.

The function $u$ is defined by choosing an arbitrary triangle
$f_{out} \in F$, excising it from the mesh, setting the values of
$u$ at two of $f_{out}$'s vertices $u_{i_1},u_{i_2}$ to $0$ and $1$,
respectively,  and then solving for the discrete harmonic $u$ that
satisfies these constraints. See for example Figure
\ref{f:discrete_type_2} (top-left); the ``missing mid-edge face''
corresponding to the excised face $f_{out}$ would have connected the
three mid-edge vertices that have a only one mid-edge face touching
them. The conjugate function $\stu$ is constructed by a simple
conjugation (and integration) process as described in
\cite{Polthier05} and \cite{Lipman:2009:MVF}.

A surprising property of the Discrete Uniformization $\Phi$ as it is
defined above, which nicely imitates the continuous theory (see
\cite{Springer57}) is that it takes the boundaries of $\M$ to
horizontal slits, see Figure \ref{f:discrete_type_2}, top row
(boundary vertices colored in red). This property allows us to
easily construct a closed form analytic map (with ``analytic'' in
its standard complex analytic sense) that will further bijectively
map the entire complex plane $\C$ minus the slit to the open unit
disk, completing our Uniformization procedure.

This property is proved by arguments similar to those for
Proposition 35 in \cite{Polthier05}; see also
\cite{Lipman:2009:MVF}. More precisely, we have
\begin{thm} \label{t:midedge_boundary}
Let $\Phi:\mM \rightarrow \C$ be the flattening map from the
mid-edge mesh $\mM$ of a mesh $\M$ with boundary, using a discrete
harmonic and conjugate harmonic pair as described above. Then, for
each connected component of the boundary of $\M$,  the mid-edge
vertices of boundary edges are all mapped onto one line segment
parallel to the real axis.
\end{thm}
\begin{proof}
Suppose $u=\sum_i u_i \phi_i(\cdot)$ is a discrete harmonic,
piecewise linear and continuous function, defined at each vertex
$v_i \in\V$,  excluding the two vertices of the excised triangle for
which values are prescribed; then we have, by
(\ref{e:partials_of_dirichlet}),
\begin{equation}\label{e:euler-lagrange_discrete_harmonic}
    \int_{R_i}\langle \nabla \phi_i, \nabla u\rangle d\vol_{\R^3} = 0,
\end{equation}

Next, consider a boundary vertex $v_j$ of the mesh $\M$. Denote by
$\mv_r,\mv_s$ the mid-edge vertices on the two boundary edges
touching vertex $v_j$. We will show that $\stu(\mv_r)=\stu(\mv_s)$;
this will imply the theorem, since $\stu$ gives the imaginary
coordinate for the images  of the mid-edge vertices under the
flattening map (see (\ref{e:Phi_flattening_map})) .

Observe that on the triangle $f_{i,j,k}$,
\begin{equation}\label{e:grad_phi_j}
    \nabla \phi_j = \frac{J(v_i-v_k)}{2\,\vol_{\Real^3}(f_{i,j,k})}.
\end{equation}

Recalling that $\nabla \stu= J \nabla u$, using
(\ref{e:grad_phi_j}), and $J^T=-J$, we obtain
\begin{align*}
\stu(\mv_r)-*u(\mv_s) & =\int_\gamma d \stu = \int_\gamma *du = \sum_{f_{i,j,k} \ni v_j} \ip{ J \nabla  u\mid_{f_{i,j,k}} , \frac{1}{2}(v_i-v_k)} \\
& = \sum_{f_{i,j,k}\ni v_j} \ip{  \nabla  u\mid_{f_{i,j,k}} , \frac{1}{2}J^T (v_i-v_k) } \\
& =  \sum_{f_{i,j,k}\ni v_j} \ip{  \nabla  u\mid_{f_{i,j,k}} , - \nabla \phi_j \mid_f } \vol_{\Real^3}(f) \\
& = - \int_M \ip{ \nabla u , \nabla \phi_j } d\vol_{\Real^3} \\
& = 0,
\end{align*}
where $\gamma$ is the piecewise linear path starting at $\mv_r$ and
passing through the mid-edge vertices of the 1-ring neighborhood of
$v_j$ ending at $\mv_s$. The last equality is due to
(\ref{e:partials_of_dirichlet}).
\end{proof}

A natural question, when dealing with any type of finite-element
approximation, concerns convergence as the mesh is refined:
convergence in what sense, and at what rate? For discrete harmonic
functions over meshes, this convergence is discussed in
\cite{Hildebrandt06,Polthier00}. Note that these convergence results
are in the weak sense; this motivated our defining the discrete
conformal factors $\mu_{\mf}$ via integrated quantities (volumes) in
Section \ref{s:the_discrete_case_implementation}.

Finally, we note that the method presented here for Discrete
Uniformization is just one option among several; other authors have
suggested other techniques; for example \cite{Gu03}. Typically, this
part of the complete algorithm described in this paper could be
viewed as a ``black box'': the remainder of the algorithm would not
change if one method of  Discrete Uniformization is replaced by
another.

\renewcommand{\thesection}{ \Alph{section}}
\section{}
\label{a:appendix C}
\renewcommand{\thesection}{\Alph{section}}
In this Appendix we prove a lemma used in the proof of Theorem
\ref{t:relaxation}.

{\bf Lemma} {\em The $N \times N$ matrices $\pi$ satisfying
\begin{equation}
     \left \{ \begin{array}{c}
             \sum_i \pi_{ij} \leq 1 \\
             \sum_j \pi_{ij} \leq 1 \\
             \pi_{ij} \geq 0 \\
             \sum_{i,j} \pi_{ij} = M < N
           \end{array}
   \right .
\label{appc:constraints}
\end{equation}
constitute a convex polytope $\P$ of which the  extremal points are
exactly those $\pi$ that satisfy all these constraints, and that
have all entries equal to either 0 or 1. }

{\em Remark}. Note that the matrices $\pi \in \P$ with all entries
in $\{0,1\}$ have exactly $M$ entries equal to 1, and all other
entries equal to zero; if one removes from these matrices all rows
and columns that consist of only zeros, what remains is a $M \times
M$ permutation matrix.

\begin{proof}
$\P$ can be considered as a subset of $\R^{N^2}$, with all entries
nonnegative, summing to $M$. The two inequalities in
(\ref{appc:constraints}) imply that the entries of any $\pi \in \P$
are bounded by 1. These inequalities can also be rewritten as the
constraint that every entry of $A \P - b \in \R^{2N}$ is non
positive, where $A$ is a $\R^{2N}\times \R^{N^2}$ matrix, and $b$ is
a  vector in $\R^{2N}$. It follows that $\P$ is a (bounded) convex
polytope in $\R^{N^2}$.

If $\pi \in \P\subset \R^{N^2}$ has entries equal to only 0 or 1,
then $\pi$ must be an extremal point of $\P$ by the following
argument. If $\pi_{\ell}=1$, and $\pi$ is a nontrivial convex
combination of $\pi^1$ and $\pi^2$ in $\P$, then
\[
\pi=\lambda\,\pi^1\,+\,(1-\lambda)\,\pi^2 \,\mbox{ with }\lambda \in
(0,1)\, \Longrightarrow 1=
\lambda\,\pi^1_{\ell}\,+\,(1-\lambda)\,\pi^2_{\ell} \,\mbox{ with }
\pi^1_{\ell}\,,\,\pi^2_{\ell}\geq 0 \Longrightarrow
\pi^1_{\ell}=\pi^2_{\ell}=1~.
\]
A similar argument can be applied for the entries of $\pi$ that are
0. It follows that we must have $\pi^1=\pi=\pi^2$, proving that
$\pi$ is extremal.

It remains thus to prove only that $\P$ has no other extremal
points. To achieve this, it suffices to prove that the extremal
points of $\P$ are all integer vectors, i.e. vectors all entries of
which are integers -- once this is established, the Lemma is proved,
since the only integer vectors in $\P$ are those with all entries in
$\{0,1\}$.

To prove that the extremal points of $\P$ are all integer vectors,
we invoke the Hoffman-Kruskal theorem (see \cite{Lovasz86}, Theorem
7C.1), which states that, given a $L\times K$ matrix $\MM$, with all
entries in $\{-1,0,1\}$, and a vector $b \in \R^L$ with integer
entries, the vertices of the polytope defined by $\{ x \in \R^K \,;
\, (\MM x)_\ell \le b_\ell \,\mbox{ for }\, \ell=1,\ldots,L\}$ are
all integer vectors in  $\R^K$ if and only if the matrix $\MM$ is
totally unimodular, i.e. if and only if  every square submatrix of
$\MM$ has determinant 1, 0 or $-1$.

We first note that (\ref{appc:constraints}) can indeed be written in
this special form. The equality $\sum_{i,j} \pi_{ij} = M$ can be
recast as the two inequalities $\sum_{i,j} \pi_{ij} \le M$ and
$-\,\sum_{i,j} \pi_{ij} \le - M$. The full system
(\ref{appc:constraints}) can then be written as $(\MM \pi)_\ell \le
b_\ell $ for $\ell=1,\ldots,L$, where $\MM$ is a $(2N+2+N^2)\times
N^2$ matrix constructed as follows. Its first $2N$ rows correspond
to the constraints on the sums over rows and columns; the entries of
the next row are all 1, and of the row after that, all $-1$ -- these
two rows correspond to the constraint $\sum_{i,j} \pi_{ij} = M$; the
final $N^2 \times N^2$ block is diagonal, with all its diagonal
entries equal to $-1$. The first $2N$ entries of $b$ are 1; the next
2 entries are $M$ and $-M$; its final $N^2$ entries are 0. By the
Hoffman-Kruskal theorem it suffices thus to show that $\MM$ is
totally unimodular.

Because the last $N^2$ rows,  the {\em bottom rows} of $\MM$, have
only one non-zero entry, which equals $-1$, we can disregard them.
Indeed, if we take a square submatrix of $\MM$ that includes (part
of) one of these bottom rows, then the determinant of the submatrix
is 0 if only zero entries of the bottom row ended up in the
submatrix; if the one -1 entry of the bottom row is an entry in the
submatrix, then the determinant is, possibly up to a sign change,
the same as if that row and the column of the $-1$ entry are
removed. By this argument, we can remove all the rows of the
submatrix partaking of the bottom rows of $\MM$.

We thus have to check unimodularity only for $\MM'$, the submatrix
of $\MM$ given by its first $2N+2$ rows. If any submatrix contains
(parts of) both the $(2N+1)$st and the $(2N+2)$nd row, then the
determinant is automatically zero, since the second of these two
rows equals the first one, multiplied by -1. This reduces the
problem to checking that $\MM''$, the submatrix of $\MM$ given by
its first $2N+1$ rows, is totally unimodular.

We now examine the top $2N$ rows of $\MM''$ more closely. A little
scrutiny reveals that it is, in fact,  the adjacency matrix $\GG$ of
the complete bipartite graph with $N$ vertices in each
part.\footnote{ The adjacency matrix $A$ for a graph $\mathcal{G}$
has as many columns as $\mathcal{G}$ has edges, and as many rows as
$\mathcal{G}$ has vertices; if we label the rows and columns of $A$
accordingly, then $A_{ve}=1$ if the vertex $v$ is an end point of
the edge $e$; otherwise $A_{ve}=0$. An adjacency matrix thus has
exactly two nonzero entries (both equal to 1) in each column. The
number of nonzero entries in the row with index $v$ is the degree of
$v$ in the graph. } It is well-known (see e.g. Theorem 8.3 in
\cite{Schrijver08}) that this adjacency matrix is  totally
unimodular, so any square submatrix of $\MM''$ that does not involve
the $(2N+1)$st row of $\MM''$ is already known to have determinant
0, 1 or $-1$. We thus have to check only submatrices that involve
the last row, i.e. matrices that consist of a $(n-1)\times n$
submatrix of $\GG$, with an added $n$th row with all entries equal
to 1. We'll denote such submatrices by $\GG'$.

We can then use a simple induction argument on $n$ to finish the
proof. The case $n=2$ is trivial. In proving the induction step for
$n=m$, we can assume that each of the top $m-1$ rows of our $m
\times m$ submatrix $\GG'$ contains at least two entries equal to 1,
since otherwise the determinant of $\GG'$ would automatically be 0,
1 or -1 by induction.

The first $m-1$ rows of $\GG'$ correspond to vertices in the
bipartite graph, and can thus be partitioned into two sets $S_1$ and
$S_2$, based on which of the two parts of $N$ vertices in the graph
they pertain to. Let us call $S$ the larger of $S_1$ and $S_2$; $S$
consists of at least $\lceil \frac{m-1}{2} \rceil$ rows.  Let us
examine the $(\# S) \times m$ sub-matrix $\GG''$ constructed from
exactly these rows.  We know that each column of $\GG''$ has exactly
one entry  $1$, since all the rows of $\GG''$ correspond to the same
group of vertices in the bipartite graph. Therefore, summing all the
rows of $\GG''$ gives a vector $v$ of only zeros and ones; since
each row in $\GG''$ contains at least two entries equal to 1, the
sum of all entries in $v$ is at least $2\parr{\lceil \frac{m-1}{2}
\rceil} \geq m-1$. The vector $v$ has thus at least $m-1$ entries
equal to $1$; the remaining $m$th entry of this linear combination
of the top $m-1$ rows of $\GG'$ is either 1 or 0. In the first case,
the determinant of $\GG'$ vanishes, since its last row also consists
of only ones. In the second case, we can subtract $v$ from the last
row of $\GG'$ without changing the value of the determinant; the
resulting last row has all entries but one equal to 0, with a
remaining entry equal to 1. The determinant is then given by the
minor of this remaining entry, and is thus 0, 1 or -1 by the
unimodularity of $\GG$.
\end{proof}

\end{document}